\documentclass[12pt,reqno]{amsart}
\usepackage{amsmath,
amsfonts,
amssymb,
amsthm,
amscd,
amsbsy}
\usepackage[usenames,
dvipsnames,
svgnames,
x11names,
hyperref]{xcolor}
\usepackage{geometry}
\usepackage{enumerate}
\usepackage{graphicx}
\usepackage{hyperref}
\usepackage{amsmath}
\usepackage{amsfonts}
\usepackage{amssymb}
\providecommand{\U}[1]
{\protect\rule{.1in}{.1in}}
\hypersetup{colorlinks=true,
breaklinks=true,
urlcolor=NavyBlue,
linkcolor=Fuchsia,
bookmarksopen=false,
filecolor=black,
citecolor=ForestGreen,
linkbordercolor=red
}
\newtheorem{theorem}{Theorem}[section]
\newtheorem{corollary}[theorem]{Corollary}
\newtheorem{lemma}[theorem]{Lemma}
\newtheorem{remark}[theorem]{Remark}

\theoremstyle{definition}

\numberwithin{equation}{section}
\allowdisplaybreaks
\AtBeginDocument{\hypersetup{pdfborder={0 0 0.01}}}

\newtheorem{theorema}{Theorem}[section]

\usepackage{cite}
\begin{document}

\title[Sharp stability for the second-order weighted HUP]{Sharp stability for the second-order weighted Heisenberg Uncertainty Principle with explicit stability constants and optimizers}
\author{Xiao-Ping Chen}
\address{Xiao-Ping Chen\newline
\indent School of Science, Xihua University, Chengdu 610039, China}
\email{xpchen@xhu.edu.cn}
\date{\today}

\begin{abstract}
By using spherical harmonicas decomposition and Gaussian-type  Poincar\'{e} inequalities, we establish several sharp stability estimates for the following  second-order weighted Heisenberg Uncertainty Principle
\begin{equation*}
\int_{\mathbb{R}^N}
\!\frac{|\Delta u|^2}
{|x|^{2\alpha}}
\mathrm{d}x
\int_{\mathbb{R}^N}
|x|^{2\alpha+2}|\nabla u|^2\mathrm{d}x
\geq \frac{(N+4\alpha+2)^2}{4}
\left(\int_{\mathbb{R}^N}
|\nabla u|^2
\mathrm{d}x\right)^2,
\end{equation*}
and
\begin{equation*}
\int_{\mathbb{R}^{N}}
\frac{|\Delta u|^{2}}
{|x|^{2\alpha}}
\mathrm{d}x
+\int_{\mathbb{R}^{N}}
\left|x\right|^{2\alpha+2}
|\nabla u|^{2}\mathrm{d}x
\ge\left(N+4\alpha+2\right)
\int_{\mathbb{R}^{N}}
|\nabla u|^{2}\mathrm{d}x.
\end{equation*}
We also provide the explicit value and the necessary and sufficient condition for attainability of the sharp stability constants. Moreover, when $\alpha=0$, our results reduce into those of [\emph{Calc. Var. Partial Differential Equations} \textbf{64} (2025), Paper No. 129], [\emph{J. Funct. Anal.} \textbf{290} (2026), Paper No. 111321] and [arXiv:2510.00453].
\end{abstract}


\maketitle

\section{Introduction and main results}

\subsection{Motivation}

We begin with the first-order Heisenberg Uncertainty Principle (HUP for short, also called Heisenberg-Pauli-Weyl Uncertainty Principle), which has important applications in quantum mechanics and  mathematical physics (see \cite{Heisenberg26,Lieb10,Weyl50}). Mathematically, for $N\ge1$ and $u\in C_c^\infty(\mathbb{R}^N)$,
\begin{equation}\label{8.11-8}
\int_{\mathbb{R}^N}
|\nabla u|^2\mathrm{d}x
\int_{\mathbb{R}^N}
|x|^2\left|u\right|^2\mathrm{d}x
\ge\frac{N^2}{4}
\left(\int_{\mathbb{R}^N}
\left|u\right|^2\mathrm{d}x\right)^2,
\end{equation}
where the constant $\frac{N^2}{4}$ in \eqref{8.11-8} is sharp and is  attained by the Gaussian functions  $u(x)=a\exp(-b|x|^2)$ with $a\in\mathbb{R}$ and $b>0$. The inequality \eqref{8.11-8} is an important subfamily of the first-order interpolation  Caffarelli-Kohn-Nirenberg inequalities (see \cite{Caffarelli84}, CKN inequalities for short), which include many important inequalities such as Sobolev inequalities, Hardy inequalities, Gagliardo-Nirenberg inequalities and so on.

In recent years, enormous researchers paid attention to quantitative stability of functional inequalities, which originated from an open question proposed by Brezis and Lieb \cite{Brezis85} regarding the stability of the classical Sobolev inequality. This open question has been answered affirmatively  by Bianchi and Egnell \cite{Bianchi91}, and established the following stability estimate
\begin{equation*}
S_{n}
\|\nabla u\|^2_{2}
-\|u\|^2_{2^*}
\ge c_{\mathrm{BE}}
\inf_{v\in E_{Sob}}
\|\nabla (u-v)\|^2_{2},
\end{equation*}
for some constant $c_{\mathrm{BE}}>0$, where $E_{Sob}$ is the set of all optimizers for Sobolev inequality. After this, \cite{Dolbeault25,Konig23} provide some explicit estimate on the stability constant $c_{\mathrm{BE}}$. For more stability results, we refer the interested readers to \cite{Carlen13,Nguyen19,
Zhang26-2,Chen25-2,Dolbeault13} for Gagliardo-Nirenberg-Sobolev inequalities, \cite{Chen25-3,Cazacu24-2} for Hardy inequalities and \cite{Cazacu24,Wei22} for CKN inequalities. It is worth mentioning that the stability version of HUP  \eqref{8.11-8} has been established widely, see  \cite{McCurdy21,Fathi21,Cazacu24,Lam26} and the reference therein.

Recently, Cazacu, Flynn and Lam in \cite{Cazacu22} generalized HUP \eqref{8.11-8} to the following second-order form.

\begin{theorema}
[{\!\rm{\!\cite[Theorem 2.1]{Cazacu22}}}]
\label{thm-a}
Let $N\ge1$. Then for any $u\in
C_c^\infty(\mathbb{R}^N)$,
\begin{equation}\label{sUP-2}
\int_{\mathbb{R}^N}
|\Delta u|^2
\mathrm{d}x
\int_{\mathbb{R}^N}
|x|^2|\nabla u|^2\mathrm{d}x
\geq \frac{(N+2)^2}{4}
\left(\int_{\mathbb{R}^N}
|\nabla u|^2\mathrm{d}x\right)^2,
\end{equation}
where the constant $\frac{(N+2)^2}{4}$ is sharp and attained by $u(x)=a\exp(-b|x|^2)$ for  $a\in\mathbb{R},b>0$.
\end{theorema}

The stability result for the second-order HUP \eqref{sUP-2} was first established by Duong and Nguyen \cite{Duong25-2}, they showed that, for $N\ge2$,
\begin{align*}
&\left(\int_{\mathbb{R}^N}
|\Delta u|^2\mathrm{d}x
\right)^{\frac{1}{2}}
\left(\int_{\mathbb{R}^N}
|x|^2|\nabla u|^2\mathrm{d}x
\right)^{\frac{1}{2}}
-\frac{N+2}{2}
\int_{\mathbb{R}^N}
|\nabla u|^2\mathrm{d}x
\\&\quad\ge
\frac{1}{384(N+2)}
\inf_{u^{\ast}\in E_{HUP}}
\left\{
\frac{\|\nabla(u-u^*)\|_2^2}
{\|\nabla u\|_2^2}:
\|\nabla u\|_2^2
=\|\nabla u^*\|_2^2
\right\},
\end{align*}
where $E_{HUP}$ is the set of all optimizers for the second order HUP \eqref{sUP-2}. Subsequently, Do, Lam and Lu in \cite{Do26-2} improved the above result to the following sharp form
\begin{align*}
&\left(\int_{\mathbb{R}^N}
|\Delta u|^2\mathrm{d}x
\right)^{\frac{1}{2}}
\left(\int_{\mathbb{R}^N}
|x|^2|\nabla u|^2\mathrm{d}x
\right)^{\frac{1}{2}}
-\frac{N+2}{2}
\int_{\mathbb{R}^N}
|\nabla u|^2\mathrm{d}x
\ge
\frac{C(N)}{2}
\inf_{u^*\in E}
\|\nabla(u-u^*)\|_2^2,
\end{align*}
and also provide some estimates for the sharp stability constant $C(N)$.
Recently, Huang and Ye in  \cite{Huang25} obtained the explicit form of the stability constant $C(N)=\sqrt{N^2+4N-4}-N$ and their optimizers.

Let $X$ be the completion of $C_c^{\infty}(\mathbb{R}^{N})$ under the norm
\[
\left(\int_{\mathbb{R}^{N}}
\frac{|\Delta u|^{2}}{|x|^{2\alpha}}
\mathrm{d}x
+\int_{\mathbb{R}^{N}}
\left|x\right|^{2\alpha+2}
|\nabla u|^{2}\mathrm{d}x\right)  ^{\frac{1}{2}}.
\]
In \cite{Cazacu23}, Cazacu, Flynn and Lam improved \eqref{sUP-2} to the following weighted version.

\begin{theorema}
[{\!\rm{\!\cite[Corollary 2.5]{Cazacu23}}}]
\label{thm-b}
Let $N\ge2$ and $\alpha\in\mathbb{R}$, there holds
\begin{equation}\label{HUP_curl_scalar}
\int_{\mathbb{R}^N}
\!\frac{|\Delta u|^2}
{|x|^{2\alpha}}
\mathrm{d}x
\int_{\mathbb{R}^N}
|x|^{2\alpha+2}|\nabla u|^2\mathrm{d}x
\geq \frac{(N+4\alpha+2)^2}{4}
\left(\int_{\mathbb{R}^N}
|\nabla u|^2
\mathrm{d}x\right)^2,
\end{equation}
for any $u\in X$. Furthermore, if $\alpha+1>0$, then
$\frac{(N+4\alpha+2)^2}{4}$ is the sharp constant of \eqref{HUP_curl_scalar} and can be
attained by the Gaussian functions
\[
u\in E_{SHUP}:=\left\{a \exp\left(-\frac{b}{2\alpha+2}
|x|^{2\alpha+2}
\right):
\
a\in\mathbb{R},
\
b>0\right\}.
\]
\end{theorema}

Inspired by the results mentioned  above, we aim to establish the stability version of the second-order weighted HUP \eqref{HUP_curl_scalar}.

\subsection{Main results}
\label{main_results}

Let
\begin{align}\label{9.18-3}
C(N,\alpha,k)
&=\inf_{u \text{ is radial}}
\frac{\int_{\mathbb{R}^{N+2k}}
\frac{|\Delta u|^{2}}{|x|^{2\alpha}}
\mathrm{d}x
+\int_{\mathbb{R}^{N+2k}}
\left[|x|^{2\alpha+2}
\left|\nabla u\right|^{2}
\mathrm{d}x
-2(\alpha+1)k
|x|^{2\alpha}|u|^{2}
\right]\mathrm{d}x}
{\int_{\mathbb{R}^{N+2k}}
|\nabla u|^{2}\mathrm{d}x}
\nonumber\\&\quad
-\left(N+4\alpha+2\right).
\end{align}
From \cite[page 1]{Huang25}, we know that under invariance of scaling, the inequality
\begin{align*}
\mu&=\inf_{u\in\mathbb{H}\setminus\{0\}}
\frac{\sqrt{H(u)U(u)}}{P(u)}
\end{align*}
is equivalent to the inequality \begin{align*}
2\mu&=\inf_{u\in\mathbb{H}\setminus\{0\}}
\frac{H(u)+U(u)}{P(u)},
\end{align*}
where $\mathbb{H}$ is a Hilbert functional space and $H,U,P$ are continuous and positive definite quadrate forms on $\mathbb{H}\setminus\{0\}$. The above argument together with  \cite[Corollary 2.5]{Cazacu23} (or see Theorem \ref{thm-b} above), we see that $C(N,\alpha,0)=0$. In view of this, we present some estimates on $C(N,\alpha,k)$ in what follows.

\begin{theorem}\label{thm-1}
Let $N\geq2$ and $\alpha>-1$, there holds
\[
C(N,\alpha,k)
=\sqrt{\left(N+2\alpha\right)^{2}
+4(N-2+k)k}-(N+2\alpha),
\]
for all $k\in\mathbb{N}^+$. As a consequence,
\[
\inf_{K\in\mathbb{N}^+}C(N,\alpha,k)
=C(N,\alpha,1)
=\sqrt{\left(N+2\alpha\right)^{2}
+4(N-1)}-(N+2\alpha).
\]
\end{theorem}

The stability result of the second-order weighted (scale non-invariant) HUP is stated below. First, let
\begin{equation*}
\delta_{1}(u)
:=\int_{\mathbb{R}^{N}}
\frac{|\Delta u|^{2}}
{|x|^{2\alpha}}
\mathrm{d}x
+\int_{\mathbb{R}^{N}}
\left|x\right|^{2\alpha+2}
|\nabla u|^{2}\mathrm{d}x
-\left(N+4\alpha+2\right)
\int_{\mathbb{R}^{N}}
|\nabla u|^{2}\mathrm{d}x
\end{equation*}
be the second-order weighted (scale non-invariant) HUP deficit.

\begin{theorem}\label{thm-2}
Let $N\geq2$ and $\alpha>-1$. Then, for all $u\in X$, we have
\begin{equation}\label{9.22-1}
\delta_{1}(u)
\geq C(N,\alpha)
\inf_{c}\int_{\mathbb{R}^{N}}
\left|\nabla\left[
u-c\exp\left(-\frac{|x|^{2\alpha+2}}
{2\alpha+2}\right)\right]\right| ^{2}\mathrm{d}x,
\end{equation}
where
\begin{equation}\label{defcna}
C(N,\alpha):=\min\left\{C(N,\alpha,1),
\
4(\alpha+1)\right\}.
\end{equation}
Furthermore,

\begin{enumerate}
\setlength{\itemsep}{0pt}
\setlength{\parsep}{0pt}
\setlength{\parskip}{2pt}

\item[(1)]
if $\alpha\geq \frac{\sqrt{N^2+4N-4}-N-6}{8}$, then $C(N,\alpha)=C(N,\alpha,1)$ is sharp and is attained if and only if the optimizer $u$ satisfies the form
\begin{equation*}
\qquad\qquad
u(x)=C|x|\,
{_1F_1}
\left(\frac{2N+4\alpha+2+C(N,\alpha)}
{4\alpha+4};
\frac{N+2\alpha+2}{2\alpha+2};
-\frac{|x|^{2\alpha+2}}{2\alpha+2}\right)
\phi_1\left(\frac{x}{|x|}\right),
\end{equation*}
where $\phi_1$ is the orthonormal
eigenfunction on the unit sphere  $\mathbb{S}^{N-1}$ satisfying  $-\Delta_{\mathbb{S}^{N-1}}\phi_1
=(N-1)\phi_1$, and ${_1F_1}(A;B;z)$ are the Kummer's confluent hypergeometric functions;

\item[(2)]
if $-1<\alpha<\frac{\sqrt{N^2+4N-4}-N-6}
{8}$, then $C(N,\alpha)=4(\alpha+1)$ is sharp and is attained if and only if the optimizer $u$ satisfies the form
\[
\qquad\quad
u(x)=-\left(\frac{2a_1}{\alpha+1}
|x|^{2\alpha+2}
+4a_1+a_0\right)
\exp\left(-\frac{|x|^{2\alpha+2}}
{2\alpha+2}\right)
+4a_1+a_0,
\]
for $a_0,a_1\in\mathbb{R}$.

\end{enumerate}
\end{theorem}

Now, let
\begin{equation*}
\delta_{2}(u):=\left(  \int_{\mathbb{R}^{N}}
\frac{|\Delta u|^{2}}{|x|^{2\alpha}}
\mathrm{d}x\right)
^{\frac{1}{2}}\left(  \int_{\mathbb{R}^{N}}
\left|x\right|^{2\alpha+2}
|\nabla u|^{2}
\mathrm{d}x\right)^{\frac{1}{2}}
-\dfrac{N+4\alpha+2}{2}
\int_{\mathbb{R}^{N}}
|\nabla u|^{2}\mathrm{d}x
\end{equation*}
be the second-order weighted (scale invariant) HUP deficit. Based on Theorem \ref{thm-2}, by the scaling argument, we obtain the following stability version of the second-order weighted (scale invariant) HUP.

\begin{theorem}\label{thm-3}
Let $N\geq2$ and $\alpha>-1$, for all $u\in X$, we have
\begin{equation*}
\delta_{2}(u)
\geq\dfrac{C(N,\alpha)}{2}
\inf_{u^{\ast}\in E_{SHUP}}
\left\|\nabla(u
-u^{\ast})\right\|_{2}^{2},
\end{equation*}
where $C(N,\alpha)/2$ is the sharp constant of the above inequality.
\end{theorem}

As a consequence of Theorem \ref{thm-3}, we obtain the following stability result.

\begin{theorem}\label{thm-4}
Let $N\geq2$ and $\alpha>-1$, for all $u\in X$, we have
\begin{equation}\label{cdtStab}
\delta_{2}(u)
\geq\dfrac{C(N,\alpha)}{4}
\inf_{u^{\ast}\in E_{SHUP}}
\left\{\left\|
\nabla(u-u^{\ast})\right\|_{2}^{2}:
\left\|\nabla u\right\|_{2}^{2}
=\left\|\nabla u^{\ast}
\right\|_{2}^{2}\right\},
\end{equation}
where $C(N,\alpha)$ is defined by \eqref{defcna}.
\end{theorem}

\begin{remark}
\rm
Compared our main results with those of \cite{Huang25,Do26-2,Duong25-2}, our advantage lies in the following several aspects.
\begin{enumerate}
\setlength{\itemsep}{0pt}
\setlength{\parsep}{0pt}
\setlength{\parskip}{2pt}

\item[(1)]
Compared with \cite{Huang25,Do26-2,Duong25-2}, in which authors investigated the stability of the second-order HUP, we establish the sharp stability of the second-order weighted HUP, namely, when $\alpha=0$, our main  results reduces into those of \cite{Huang25,Do26-2,Duong25-2}.

\item[(2)]
We also provide the explicit value of the sharp stability constants, and the sufficient and necessary condition for the attainability of the sharp stability constants.

\end{enumerate}
\end{remark}

Now, we briefly state main difficulties and strategy during the proof of our main results. The proof of Theorems \ref{thm-3} and \ref{thm-4} heavily depended on the result of Theorem \ref{thm-2}, thus it is sufficient to consider Theorem \ref{thm-2}.

Our approach is different from that of \cite{Huang25,Do26-2,Duong25-2}. More precisely, Duong and Nguyen in \cite{Duong25-2} apply the following identity
\begin{align*}
&\int_{\mathbb{R}^{N}}
|\Delta u|^{2}\mathrm{d}x
+\int_{\mathbb{R}^{N}}
\left|x\right|^{2}
|\nabla u|^{2}\mathrm{d}x
-\left(N+2\right)
\int_{\mathbb{R}^{N}}
|\nabla u|^{2}\mathrm{d}x
\\&\quad=
\int_{\mathbb{R}^{N}}
\left\|\nabla^2 v
-x\otimes\nabla v
\right\|^{2}_{HS}
e^{-|x|^2}\mathrm{d}x,
\end{align*}
where $v=ue^{\frac{|x|^2}{2}}$,  $\nabla^2v$ denotes the Hessian matrix of $v$, $x\otimes\nabla v$ is the matrix $(x_i\partial_jv)_{i,j}$ and $\|A\|_{HS}$ is the Hilbert-Schmidt norm of the matrix $A$, then using spectral analysis of the Ornstein-Uhlenbeck operator associated with the Gaussian-type measure and Hermite polynomials to consider the right side of the above identity. However, the process is very complicated due to the term $\left\|\nabla^2 v
-x\otimes\nabla v
\right\|^{2}_{HS}$, and also fail to obtain the sharpness of the stability constants. Subsequently, Do, Lam and Lu in \cite{Do26-2} apply spherical harmonics decomposition and harmonic analysis to establish sharp stability of the second-order HUP \eqref{sUP-2}. The key idea is to apply Fourier transform
\begin{equation}\label{9.23-2}
\int_{\mathbb{R}^{N+2}}
|\Delta^{\frac{\iota}{2}} u|^{2}\mathrm{d}x
=\frac{|\mathbb{S}^{N+1}|}
{|\mathbb{S}^{N-1}|}
\int_{\mathbb{R}^{N}}
|\Delta^{\frac{\iota+1}{2}} u|^{2}\mathrm{d}x
\end{equation}
to convert the second-order HUP on $\mathbb{R}^{N+2k}$ to the first-order HUP on  $\mathbb{R}^{N+2k+2}$. However, the explicit value of the optimal stability constant has not been provided in \cite{Do26-2}. Recently, Huang and Ye in \cite{Huang25} give the explicit value of the sharp stability constant and also provide optimizers of the sharp stability constant.

In this paper, we also apply the standard spherical harmonics decomposition to \eqref{9.22-1}, and rewrite
\begin{equation*}
u(x)=u(r\sigma)
=\sum_{k=0}^{\infty}
r^{k}v_{k}(r)
\phi_{k}(\sigma).
\end{equation*}
However, since the presence of the weight $|x|^{2\alpha}$, the above Fourier transform \eqref{9.23-2} is not valid, so we apply the transformation $w_0(r)
=v_0^{\prime}(r)/r^{2\alpha+1}$ (for $\alpha>-1$) to obtain
\begin{align*}
\int_{\mathbb{R}^{N}}
\frac{|\Delta v_0(|x|)|^{2}}
{|x|^{2\alpha}}\mathrm{d}x
&=\int_0^\infty
|w'_0|^{2}r^{N+2\alpha+1}
\mathrm{d}x,\\
\int_{\mathbb{R}^{N}}
|x|^{2\alpha+2}|\nabla u|^{2}
\mathrm{d}x
&=\int_0^\infty
|w_0|^{2}
r^{N+6\alpha+3}
\mathrm{d}r,\\
\int_{\mathbb{R}^{N}}
\frac{|\Delta v_0(|x|)|^{2}}
{|x|^{2\alpha}}\mathrm{d}x
&=\int_0^\infty
|w_0|^{2}r^{N+4\alpha+1}
\mathrm{d}r.
\end{align*}
In view of this, \eqref{9.22-1} is equivalent to
\begin{align*}
&\int_0^\infty
|w'_0|^{2}
r^{N+2\alpha+1}\mathrm{d}r
+\int_0^\infty
|w_0|^{2}
r^{N+6\alpha+3}
\mathrm{d}r
-(N+4\alpha+2)
\int_0^\infty
|w_0|^{2}r^{N+4\alpha+1}
\mathrm{d}r
\nonumber\\&\qquad
+\sum_{k=1}^{\infty}
\int_{\mathbb{R}^{N+2k}}
\frac{|\Delta v_k|^{2}}
{|x|^{2\alpha}}
\mathrm{d}x
+\sum_{k=1}^{\infty}
\int_{\mathbb{R}^{N+2k}}
\left[|x|^{2\alpha+2}
|\nabla v_k|^{2}
-2(\alpha+1)k
|x|^{2\alpha}|v_k|^{2}\right]
\mathrm{d}x
\nonumber\\&\qquad
-[N+4\alpha+2+C(N,\alpha)]
\sum_{k=1}^{\infty}
\int_{\mathbb{R}^{N+2k}}
|\nabla v_k|^{2}\mathrm{d}x
\\&\quad\ge
C(N,\alpha)
\inf_{c}\int_0^\infty
\left|w_0-c
\exp\left(-\frac{r^{2\alpha+2}}
{2\alpha+2}\right)\right|^2
r^{N+4\alpha+1}\mathrm{d}r.
\end{align*}
We then consider the case $k=0$ and the case $k\ge1$ separately. More precisely, we apply Gaussian-type Poincar\'{e} inequalities (see Corollary \ref{coro-5.2} below) to handle the case $k=0$:
\begin{align*}
&\int_0^\infty
|w'_0|^{2}
r^{N+2\alpha+1}\mathrm{d}r
+\int_0^\infty
|w_0|^{2}
r^{N+6\alpha+3}
\mathrm{d}r
-(N+4\alpha+2)
\int_0^\infty
|w_0|^{2}r^{N+4\alpha+1}
\mathrm{d}r
\\&\quad\ge
C(N,\alpha)
\inf_{c}\int_0^\infty
\left|w_0-c
\exp\left(-\frac{r^{2\alpha+2}}
{2\alpha+2}\right)\right|^2
r^{N+4\alpha+1}\mathrm{d}r.
\end{align*}
We define a deficit  $C(N,\alpha,k)$ and consider some properties of it to handle the case $k\ge1$:
\begin{align*}
&\int_{\mathbb{R}^{N+2k}}
\frac{|\Delta v_k|^{2}}
{|x|^{2\alpha}}
\mathrm{d}x
+\int_{\mathbb{R}^{N+2k}}
\left[|x|^{2\alpha+2}
|\nabla v_k|^{2}
-2(\alpha+1)k
|x|^{2\alpha}|v_k|^{2}\right]
\mathrm{d}x
\nonumber\\&\quad
\ge[N+4\alpha+2+C(N,\alpha)]
\int_{\mathbb{R}^{N+2k}}
|\nabla v_k|^{2}\mathrm{d}x.
\end{align*}
More importantly, inspired by \cite{Huang25}, applying the change of variables to Kummer's ODEs yields explicit value of the sharp stability constants $C(N,\alpha)$, we also provide the sufficient and necessary condition for the attainability of the sharp stability constants $C(N,\alpha)$.

\subsection{Structure of the rest of this paper}\label{sect-1.3}

In Section \ref{sect-2}, we provide some preliminary facts, including some notations and the spherical harmonics decomposition technique. In Section \ref{sect-4}, we consider stability estimates of the second-order weighted HUP, and prove Theorems \ref{thm-1}, \ref{thm-2}, \ref{thm-3} and \ref{thm-4}.

\section{Preliminaries}\label{sect-2}

We start with the following notations which will be used frequently in this paper.
\begin{itemize}
\setlength{\itemsep}{0pt}
\setlength{\parsep}{0pt}
\setlength{\parskip}{2pt}

\item
$\mathbb{N}:=\{0,1,2,\cdots\}$ denotes the set which includes all natural numbers.

\item
$\mathbb{N}^+:=\{1,2,\cdots\}$ denotes the set which includes all positive integers.

\item
For all $p\in[1,\infty)$, $L^p(\mathbb{R}^N)$ is the usual Lebesgue space with the norm
\[
\|u\|_p:=\left(\int_{\mathbb{R}^N}
|u|^p\mathrm{d}x\right)^{\frac{1}{p}}.
\]

\end{itemize}

The main idea of this paper is to adapt the well-known technique of decomposing a function $u$ into spherical harmonicas, which is a useful method (see \cite{Cazacu20,Cazacu22,Duong25-2,
Huang25,Tertikas07,Vazquez00,Do26-2} and the references therein).
For $N\ge2$, we apply the coordinate
transformation $x\in\mathbb{R}^N
\mapsto(r,\sigma)
\in(0,\infty)\times
\mathbb{S}^{N-1}$, and then for each
$u\in C^\infty_c(\mathbb{R}^N)$,
\begin{equation*}
u(x)=u(r\sigma)
=\sum^\infty_{k=0}
u_k(r)\phi_k(\sigma),
\end{equation*}
where $\phi_k$ are the orthonormal
eigenfunctions of the Laplace-Beltrami
operator on $\mathbb{S}^{N-1}$ with the corresponding eigenvalue $c_k=k(N+k-2)$, namely, $-\Delta_{\mathbb{S}^{N-1}}\phi_k
=c_k\phi_k$ with $k\in\mathbb{N}$.
The Fourier coefficients $\{u_k\}_k$ belong to
$C_c^\infty([0,\infty))$ and satisfy
$u_{k}(r)=O(r^k)$ as $r\to 0$, which
allows us to consider the change of variables
\begin{equation*}
u_{k}(r)=r^kv_k(r),
\end{equation*}
where $v_k\in C^\infty_c([0,\infty))$
(see \cite[page 418]{Tertikas07} for more details). Thus, the function $u$ can be represented by the following form
\begin{equation*}
u(x)=u(r\sigma)
=\sum_{k=0}^{\infty}u_{k}(r)
\phi_{k}(\sigma)
=\sum_{k=0}^{\infty}r^{k}v_{k}(r)
\phi_{k}(\sigma).
\end{equation*}
Then, direct calculations together with \cite[(2.3)--(2.5)]{Duong25-2}  imply the following identities
\begin{align}\label{9.12-2}
\int_{\mathbb{R}^{N}}
|\nabla u|^{2}\mathrm{d}x
&=\sum_{k=0}^{\infty}
\int_{\mathbb{R}^{N+2k}}
\left|\nabla v_{k}(|x|)\right|^{2}
\mathrm{d}x
\nonumber\\&=\int_{\mathbb{R}^{N}}
\left|\nabla v_0(|x|)\right|^{2}
\mathrm{d}x
+\sum_{k=1}^{\infty}
\int_{\mathbb{R}^{N+2k}}
\left|\nabla v_{k}(|x|)\right|^{2}
\mathrm{d}x
\nonumber\\&=
\int_{0}^{\infty}
\left|v_0'\right|^{2} r^{N-1}\mathrm{d}r
+\sum_{k=1}^{\infty}
\int_{\mathbb{R}^{N+2k}}
\left|\nabla v_{k}(|x|)\right|^{2}
\mathrm{d}x
\nonumber\\&=
\int_{0}^{\infty}
\left|w_0\right|^{2}
r^{N+4\alpha+1}\mathrm{d}r
+\sum_{k=1}^{\infty}
\int_{\mathbb{R}^{N+2k}}
\left|\nabla v_{k}(|x|)\right|^{2}
\mathrm{d}x,
\\
\int_{\mathbb{R}^{N}}
\frac{|\Delta u|^{2}}{|x|^{2\alpha}}\mathrm{d}x
&=\sum_{k=0}^{\infty}
\int_{\mathbb{R}^{N+2k}}
\frac{|\Delta v_k(|x|)|^{2}}
{|x|^{2\alpha}}\mathrm{d}x
\nonumber\\&=\int_{\mathbb{R}^{N}}
\frac{|\Delta v_0(|x|)|^{2}}
{|x|^{2\alpha}}\mathrm{d}x
+\sum_{k=1}^{\infty}
\int_{\mathbb{R}^{N+2k}}
\frac{|\Delta v_k(|x|)|^{2}}
{|x|^{2\alpha}}\mathrm{d}x
\nonumber\\&=\int_{0}^{\infty}
\left|v_0''
+\frac{N-1}{r}v_0'\right|^{2}
r^{N-2\alpha-1}\mathrm{d}r
+\sum_{k=1}^{\infty}
\int_{\mathbb{R}^{N+2k}}
\frac{|\Delta v_k(|x|)|^{2}}
{|x|^{2\alpha}}\mathrm{d}x
\nonumber\\&=
\int_{0}^{\infty}
\left|v_0''\right|^{2}
r^{N-2\alpha-1}\mathrm{d}r
+(2\alpha+1)(N-1)
\int_{0}^{\infty}
\left|v_0'\right|^{2}
r^{N-2\alpha-3}
\mathrm{d}r
\nonumber\\&\quad
+\sum_{k=1}^{\infty}
\int_{\mathbb{R}^{N+2k}}
\frac{|\Delta v_k(|x|)|^{2}}
{|x|^{2\alpha}}\mathrm{d}x
\nonumber\\&=\int_0^\infty
|w'_0|^{2}r^{N+2\alpha+1}
\mathrm{d}x
+\sum_{k=1}^{\infty}
\int_{\mathbb{R}^{N+2k}}
\frac{|\Delta v_k(|x|)|^{2}}
{|x|^{2\alpha}}\mathrm{d}x,\label{9.12-3}
\end{align}
and
\begin{align}\label{9.12-4}
&\int_{\mathbb{R}^{N}}
|x|^{2\alpha+2}|\nabla u|^{2}
\mathrm{d}x
\nonumber\\&\quad=%0
\sum_{k=0}^{\infty}
\int_{\mathbb{R}^{N+2k}}
\left[|x|^{2\alpha+2}
|\nabla v_k(|x|)|^{2}
-2(\alpha+1)k
|x|^{2\alpha}|v_k(|x|)|^{2}
\right]\mathrm{d}x
\nonumber\\&\quad=%2
\int_{0}^{\infty}
\left|v_0'\right|^{2}
r^{N+2\alpha+1}
\mathrm{d}r
+\sum_{k=1}^{\infty}
\int_{\mathbb{R}^{N+2k}}
\left[|x|^{2\alpha+2}
|\nabla v_k|^{2}
-2(\alpha+1)k
|x|^{2\alpha}|v_k|^{2}
\right]\mathrm{d}x
\nonumber\\&\quad=%3
\int_{0}^{\infty}
\left|w_0\right|^{2}
r^{N+6\alpha+3}
\mathrm{d}r
+\sum_{k=1}^{\infty}
\int_{\mathbb{R}^{N+2k}}
\left[|x|^{2\alpha+2}
|\nabla v_k|^{2}
-2(\alpha+1)k
|x|^{2\alpha}|v_k|^{2}
\right]\mathrm{d}x.
\end{align}
where $w_0(r)
=\dfrac{v_0'(r)}{r^{2\alpha+1}}$
 for $\alpha>-1$.

\section{Stability estimates of the second-order weighted HUP: proof of Theorem \ref{thm-1}, Theorem \ref{thm-2}, Theorem \ref{thm-3} and Theorem \ref{thm-4}}\label{sect-4}

In this section, we will establish the stability version of the second-order weighted scale invariant and scale non-invariant HUP.

\subsection{Some useful lemmas}

We first recall some elementary facts about the Kummer's confluent hypergeometric functions
\[
{_1F_1}(m;n;t)
=\sum_{j=0}^\infty
\frac{\Gamma(j+m)t^j}{\Gamma(j+n)j!},
\]
for $m,n\ge0$ and $t\in\mathbb{R}$. From \cite[pages 2-4]{Buchholz69}, the  Kummer's ODE
\begin{equation*}
tH{''}(t)
+nH'(t)
-tH'(t)
-mH(t)=0
\end{equation*}
has the following two different solutions
\begin{align*}
H_1(t)&={_1F_1}(m;n;t),\\
H_2(t)&=t^{1-n}{_1F_1}(m-n+1;2-n;t).
\end{align*}

In view of this, we use the change of variables $t\mapsto -\frac{t^{2\alpha+2}}{2\alpha+2}$ to obtain the following result.

\begin{lemma}\label{lem-3.1}
Assume that $\alpha>-1$, for $m,n\ge0$ and $t\in\mathbb{R}$, then
\begin{equation*}
t\Psi{''}(t)
+\left[(2\alpha+2)n
-(2\alpha+1)\right]\Psi'(t)
+t^{2\alpha+2}\Psi'(t)
+(2\alpha+2)m
t^{2\alpha+1}\Psi(t)=0
\end{equation*}
has two independent solutions
\begin{align*}
\Psi_1(t)&={_1F_1}
\left(m;n;
-\frac{t^{2\alpha+2}}
{2\alpha+2}\right),\\
\Psi_2(t)&
=\left(-\frac{t^{2\alpha+2}}
{2\alpha+2}\right)^{1-n}
{_1F_1}
\left(m-n+1;2-n;
-\frac{t^{2\alpha+2}}
{2\alpha+2}\right).
\end{align*}
Moreover, for all $l\in\mathbb{N}$, there holds
\begin{align}
\Psi_1^{(l)}(t)
&=O\left(t^{-(2\alpha+2)m-l}\right),
\ \
\mathrm{as}
\
t\to\infty.\label{9.19-7}
\end{align}
\end{lemma}

\begin{proof}
[\rm\textbf{Proof}]
Evidently, $\Psi(t)$ is smooth for $t\in\mathbb{R}$. Notice that, let $s:=-\frac{t^{2\alpha+2}}
{2\alpha+2}$,
\begin{align*}
&t\Psi''_1(t)
+\left[(2\alpha+2)n-(2\alpha+1)\right]
\Psi'_1(t)
+t^{2\alpha+2}\Psi'_1(t)
+(2\alpha+2)mt^{2\alpha+1}\Psi_1(t)
\\&\quad=%1
t^{4\alpha+3}H''_1(s)
-(2\alpha+1)t^{2\alpha+1}H'_1(s)
-\left[(2\alpha+2)n-(2\alpha+1)\right]
t^{2\alpha+1}H'_1(s)
\\&\qquad
-t^{4\alpha+3}H'_1(s)
+(2\alpha+2)mt^{2\alpha+1}H_1(s)
\\&\quad=%2
t^{2\alpha+1}\left[
t^{2\alpha+2}H''_1(s)
-(2\alpha+2)nH'_1(s)
-t^{2\alpha+2}H'_1(s)
+(2\alpha+2)mH_1(s)\right]
\\&\quad=%3
t^{2\alpha+1}\left[
-(2\alpha+2)sH''_1(s)
-(2\alpha+2)nH'_1(s)
+(2\alpha+2)sH'_1(s)
+(2\alpha+2)mH_1(s)\right]
\\&\quad=%4
-(2\alpha+2)t^{2\alpha+1}
\left[
sH''_1(s)
+nH'_1(s)
-sH'_1(s)
-mH_1(s)\right]
\\&\quad=0,
\end{align*}
using the fact that $sH''_1(s)
+nH'_1(s)
-sH'_1(s)
-mH_1(s)=0$. Similarly,
\[
t\Psi''_2(t)
+\left[(2\alpha+2)n-(2\alpha+1)\right]
\Psi'_2(t)
+t^{2\alpha+2}\Psi'_2(t)
+(2\alpha+2)mt^{2\alpha+1}\Psi_2(t)=0.
\]

As $t\to-\infty$, it follows from \cite[13.1.5]{Abramowite72} that ${_1F_1}(m;n;t)\sim (-t)^{-m}$ if $m\neq n$. Notice that $H_1^{(l)}(t)={_1F_1}(m+l;n+l;t)$ for all $l\in\mathbb{N}$ (see for instance, \cite[(8.80)]{Viola16},  \cite[13.4.9]{Abramowite72}), and then
\[
H_1^{(l)}(s)
={_1F_1}(m+l;n+l;s)
\sim t^{-(2\alpha+2)(m+l)},
\]
as $t\to\infty$. By induction, as $t\to\infty$,
\begin{align*}
\Psi'_1(t)&=
-t^{2\alpha+1}
H'_1(s)
\sim t^{-(2\alpha+2)m-1},\\%1
\Psi''_1(t)&
=t^{4\alpha+2}
H''_1(s)
-(2\alpha+1)t^{2\alpha}
H'_1(s)
\sim t^{-(2\alpha+2)m-2},\\%2
\Psi'''_1(t)&
=-t^{6\alpha+3}
H'''_1(s)
+(6\alpha+3)t^{4\alpha+1}
H''_1(s)
-2\alpha(2\alpha+1)t^{2\alpha-1}H'_1(s)
\sim t^{-(2\alpha+2)m-3},\\%3
&\vdots\\
\Psi^{(l)}_1(t)&\sim
t^{-(2\alpha+2)m-l},%l
\end{align*}
as our desired estimate \eqref{9.19-7}. The proof is completed here.
\end{proof}

We recall a sharp Gaussian-type Poincar\'{e} inequalities given by Do \emph{et al.} in \cite{Do26}, which are crucial to prove Theorem \ref{thm-2}.

\begin{lemma}
[{\rm{\!\cite[Theorem 2.1]{Do26}}}]\label{lem-5.2}
Let $m>0$, for any function $u\in W^{1,2}((0,\infty),
e^{-\frac{s^2}{2}}
s^{m-1}\mathrm{d}s)$, there holds
\begin{equation}\label{9.23-1}
\int_{0}^\infty
|u'(s)|^{2}
e^{-\frac{s^2}{2}}
s^{m-1}
\mathrm{d}s
\ge 2
\inf_{c}\int_0^\infty
|u(s)-c|^{2}
e^{-\frac{s^2}{2}}
s^{m-1}\mathrm{d}s,
\end{equation}
where the constant $2$ is sharp and is  attained if and only if the optimizer $u$ satisfies
\[
u(s)=a_0+a_1s^2.
\]
\end{lemma}

Inspired by \cite{Do26}, let $s=\sqrt{\frac{2}{\alpha+1}}
r^{\alpha+1}$, then we have the following corollary.

\begin{corollary}\label{coro-5.2}
Let $\alpha>-1$ and $m>0$, there holds
\begin{align*}
&\int_{0}^\infty
\left|w'(r)\right|^{2}
\exp\left(-\frac{r^{2\alpha+2}}
{\alpha+1}\right)
r^{m(\alpha+1)-2\alpha-1}\mathrm{d}r
\nonumber\\&\quad
\ge 4(\alpha+1)\inf_{c}
\int_{0}^\infty
\left|w(r)-c\right|^{2}
\exp\left(-\frac{r^{2\alpha+2}}
{\alpha+1}\right)
r^{m(\alpha+1)-1}\mathrm{d}r,
\end{align*}
where $4(\alpha+1)$ is sharp and is attained if and only if the optimizer $w$ satisfies
\[
w(r)=a_0+\frac{2a_1}{\alpha+1}r^{2\alpha+2},
\]
for $a_0,a_1\in\mathbb{R}$.
\end{corollary}

\begin{proof}
[\rm\textbf{Proof}]
Set $s=\sqrt{\frac{2}{\alpha+1}}
r^{\alpha+1}$, which implies that
$\mathrm{d}s
=\sqrt{2\alpha+2}r^{\alpha}\mathrm{d}r$.
Let us denote
\[
u\left(\sqrt{\frac{2}{\alpha+1}}
r^{\alpha+1}\right)
=w(r),
\]
and then
\[
u'\left(\sqrt{\frac{2}{\alpha+1}}
r^{\alpha+1}\right)
=\frac{1}{\sqrt{2\alpha+2}}
w'(r)r^{-\alpha}.
\]
These imply that
\begin{align*}
&\int_{0}^\infty
|u'(s)|^{2}
e^{-\frac{s^2}{2}}
s^{m-1}
\mathrm{d}s
\\&\quad=%1
\int_{0}^\infty
\left|u'\left(\sqrt{\frac{2}{\alpha+1}}
r^{\alpha+1}\right)\right|^{2}
\exp\left(-\frac{r^{2\alpha+2}}
{\alpha+1}\right)
\left(\sqrt{\frac{2}{\alpha+1}}
r^{\alpha+1}\right)^{m-1}
\mathrm{d}\left(\sqrt{\frac{2}{\alpha+1}}
r^{\alpha+1}\right)
\\&\quad=%2
\frac{1}{2\alpha+2}
\int_{0}^\infty
\left|w'(r)\right|^{2}
r^{-2\alpha}
\exp\left(-\frac{r^{2\alpha+2}}
{\alpha+1}\right)
\left(\sqrt{\frac{2}{\alpha+1}}
r^{\alpha+1}\right)^{m-1}
\sqrt{2\alpha+2}r^{\alpha}
\mathrm{d}r
\\&\quad=%3
2^{\frac{m-2}{2}}
(\alpha+1)^{-\frac{m}{2}}
\int_{0}^\infty
\left|w'(r)\right|^{2}
\exp\left(-\frac{r^{2\alpha+2}}
{\alpha+1}\right)
r^{m(\alpha+1)-2\alpha-1}\mathrm{d}r,
\end{align*}
and
\begin{align*}
&\int_0^\infty
|u(s)-c|^{2}
e^{-\frac{s^2}{2}}
s^{m-1}\mathrm{d}s
\\&\quad=%1
\int_{0}^\infty
\left|u\left(\sqrt{\frac{2}{\alpha+1}}
r^{\alpha+1}\right)-c\right|^{2}
\exp\left(-\frac{r^{2\alpha+2}}
{\alpha+1}\right)
\left(\sqrt{\frac{2}{\alpha+1}}
r^{\alpha+1}\right)^{m-1}
\mathrm{d}\left(\sqrt{\frac{2}{\alpha+1}}
r^{\alpha+1}\right)
\\&\quad=%2
\int_{0}^\infty
\left|w(r)-c\right|^{2}
\exp\left(-\frac{r^{2\alpha+2}}
{\alpha+1}\right)
\left(\sqrt{\frac{2}{\alpha+1}}
r^{\alpha+1}\right)^{m-1}
\sqrt{2\alpha+2}r^{\alpha}\mathrm{d}r
\\&\quad=%3
2^{\frac{m}{2}}
(\alpha+1)^{\frac{2-m}{2}}
\int_{0}^\infty
\left|w(r)-c\right|^{2}
\exp\left(-\frac{r^{2\alpha+2}}
{\alpha+1}\right)
r^{m(\alpha+1)-1}\mathrm{d}r.
\end{align*}
Hence, \eqref{9.23-1} is equivalent to
\begin{align*}
&2^{\frac{m-2}{2}}
(\alpha+1)^{-\frac{m}{2}}
\int_{0}^\infty
\left|w'(r)\right|^{2}
\exp\left(-\frac{r^{2\alpha+2}}
{\alpha+1}\right)
r^{m(\alpha+1)-2\alpha-1}\mathrm{d}r
\nonumber\\&\quad
\ge2^{\frac{m+2}{2}}
(\alpha+1)^{\frac{2-m}{2}}\inf_{c}
\int_{0}^\infty
\left|w(r)-c\right|^{2}
\exp\left(-\frac{r^{2\alpha+2}}
{\alpha+1}\right)
r^{m(\alpha+1)-1}\mathrm{d}r,
\end{align*}
that is,
\begin{align*}
&\int_{0}^\infty
\left|w'(r)\right|^{2}
\exp\left(-\frac{r^{2\alpha+2}}
{\alpha+1}\right)
r^{m(\alpha+1)-2\alpha-1}\mathrm{d}r
\nonumber\\&\quad
\ge 4(\alpha+1)\inf_{c}
\int_{0}^\infty
\left|w(r)-c\right|^{2}
\exp\left(-\frac{r^{2\alpha+2}}
{\alpha+1}\right)
r^{m(\alpha+1)-1}\mathrm{d}r,
\end{align*}
as our desired estimate. Here the proof is completed.
\end{proof}

Recall that
\[
E_{SHUP}=\left\{a \exp\left(-\frac{b}{2\alpha+2}
|x|^{2\alpha+2}
\right):
\
a\in\mathbb{R},
\
b>0\right\}
\]
is the set of all optimizers for the second-order weighted HUP.

\begin{lemma}\label{PropOfE}
Let $N\ge2$ and $\alpha>-1$. Then the set $E_{SHUP}$ is closed under the seminorm $\|\nabla(\cdot)\|_{2}$. Moreover, for each $u\in X$, there exists $v\in X$ such that
\begin{equation*}
\inf_{u^{\ast}\in E_{SHUP}}
\left\|
\nabla(u- u^{\ast})\right\|_{2}^{2}
=\left\|\nabla(u-v)\right\|_{2}^{2}.
\end{equation*}
\end{lemma}

\begin{proof}
[\rm\textbf{Proof}]
Let $u\in X$, assume that there exists a sequence $\{v_{j}\}_{j}
\subset E_{SHUP}$ such that
\begin{equation*}
\lim_{j\to\infty}
\|\nabla(u-v_{j})\|_{2}=0.
\end{equation*}
We will show that $u\in$ $E_{SHUP}$. If $\|\nabla u\|_{2}=0$, then $u=0$ and $u\in$ $E_{SHUP}$.
Thus, without loss of generality, we can assume that $\|\nabla u\|_{2}=1$. Using the triangle inequality, we get
\begin{align*}
\|\nabla(u-v_{j})\|_{2}
&\geq\big|\|\nabla u\|_{2}
-\|\nabla v_{j}\|_{2}\big|
=\big|1-\|\nabla v_{j}\|_{2}\big|,
\end{align*}
and then
\begin{equation*}
\dfrac{1}{2}
\leq\|\nabla v_{j}\|_{2}
\leq\dfrac{3}{2},
\end{equation*}
for $j$ large enough. Let $v_{j}=a_{j}
\exp\left(-\frac{b_{j}}
{2\alpha+2}|x|^{2\alpha+2}\right)$, a direct calculation shows that
\begin{align}\label{9.20-1}
\|\nabla v_{j}\|^{2}_2
&=%1
a_j^{2}b_j^{2}
\int_{\mathbb{R}^{N}}
|x|^{4\alpha+2}\exp\left(-\frac{b_j}
{\alpha+1}|x|^{2\alpha+2}\right)
\mathrm{d}x
\nonumber\\&=%2
a_j^{2}b_j^{\frac{2-N}{2\alpha+2}}
\int_{\mathbb{R}^{N}}
|x|^{4\alpha+2}\exp\left(-\frac{1}
{\alpha+1}|x|^{2\alpha+2}\right)
\mathrm{d}x.
\end{align}
This implies that there exist constants $M,m>0$ such that
\begin{equation}\label{relationab}
m^{2}\leq a_j^{2}
b_j^{\frac{2-N}{2\alpha+2}}
\leq M^{2},
\end{equation}
for $j$ large enough. \emph{We claim that $\{a_j\}_{j}$ and
$\{b_j\}_{j}$ are bounded}. Going if necessary to a subsequence, there exist $a\in\mathbb{R}$ and $b
\in(0,\infty)$ such that  $a_{j}\rightarrow a$ and $b_j\rightarrow b$. Let us define $v(x)=a \exp\left(-\frac{b}{2\alpha+2}
|x|^{2\alpha+2}\right)$, from the dominated convergence theorem, $\nabla v_{j}\rightarrow\nabla v$ in $L^{2}
(\mathbb{R}^{N})$. Hence, using the triangle inequality, we get
\[
\|\nabla(u-v)\|_{2}
\leq\|\nabla(u-v_{j})\|_{2}
+\|\nabla(v-v_{j})\|_{2}
\rightarrow0,
\ \
j\rightarrow\infty.
\]
Then, $u-v$ is a constant a.e. in $\mathbb{R}^{N}.$
Since $u\in X$, $u-v=0$ or $u=v=a \exp\left(-\frac{b}{2\alpha+2}
|x|^{2\alpha+2}\right)\in
E_{SHUP}$. In other words, $E_{SHUP}$ is closed under the seminorm
$\|\nabla(\cdot)\|_{2}$. Moreover, for each $u\in X$, there exists $v\in X$ such that $\inf_{u^{\ast}\in E_{SHUP}}
\left\|\nabla(u- u^{\ast})\right\|_{2}^{2}
=\left\|\nabla(u-v)\right\|_{2}^{2}$.

In view of this, to complete the proof of this lemma, it remains to show that the above claim holds. Inspired by the argument as in \cite{Duong25-2,Do26-2},  it is not difficult to verify this claim. First, it follows from \eqref{9.20-1} that
\begin{align*}
\|\nabla (u-v_{j})\|^{2}_2
&=1
+a_j^{2}
b_j^{\frac{2-N}{2\alpha+2}}
\int_{\mathbb{R}^{N}}
|x|^{4\alpha+2}\exp\left(-\frac{1}
{\alpha+1}|x|^{2\alpha+2}\right)
\mathrm{d}x
\\&\quad+2a_jb_j
\int_{\mathbb{R}^{N}}\nabla u\cdot\left(|x|^{2\alpha }x\right)
\exp\left(-\frac{b_j}{2\alpha+2}
|x|^{2\alpha+2}\right)
\mathrm{d}x
\\&=%2
1
+a_j^{2}
b_j^{\frac{2-N}{2\alpha+2}}
\int_{\mathbb{R}^{N}}
|x|^{4\alpha+2}\exp\left(-\frac{1}
{\alpha+1}|x|^{2\alpha+2}\right)
\mathrm{d}x
\\&\quad+2a_j
b_j^{\frac{1}{2\alpha+2}}
\int_{\mathbb{R}^{N}}\nabla u\cdot\left(b_j^{\frac{2\alpha+1}
{2\alpha+2}}
|x|^{2\alpha}x\right)
\exp\left(-\frac{1}{2\alpha+2}
\left|b_j^{\frac{1}{2\alpha+2}}x
\right|^{2\alpha+2}\right)
\mathrm{d}x.
\end{align*}
Up to a subsequence, we divide the proof into the following three cases:
\begin{enumerate}
\setlength{\itemsep}{0pt}
\setlength{\parsep}{0pt}
\setlength{\parskip}{2pt}

\item[(1)]
$\lim_{j \rightarrow\infty} b_j=\infty$;

\item[(2)]
$\lim_{j \rightarrow\infty} b_j=0$;

\item[(3)]
$\lim_{j\rightarrow\infty}b_j \in(0,\infty)$.
\end{enumerate}

If Case (1) holds, that is,
$\lim_{j\rightarrow\infty}b_j=\infty$. For any $\epsilon
>0$, there exists a constant $R>0$ small enough such that $\int_{B_{R}}|\nabla
u|^{2}\mathrm{d}x<\epsilon^{2}$, where $B_R:=\{x\in\mathbb{R}^{N}:|x|<R\}$. Then, using H\"{o}lder inequality, we see that
\begin{align*}
&\left|\int_{\mathbb{R}^{N}}\nabla u\cdot\left(b_j^{\frac{2\alpha+1}
{2\alpha+2}}
|x|^{2\alpha }x\right)
\exp\left(-\frac{1}{2\alpha+2}
\left|b_j^{\frac{1}{2\alpha+2}}x
\right|^{2\alpha+2}\right)
\mathrm{d}x\right|
\\&\quad\leq%1
\int_{B_{R}}
|\nabla u|
\left|b_j^{\frac{1}{2\alpha+2}}
x\right|^{2\alpha+1}
\exp\left(-\frac{1}{2\alpha+2}
\left|b_j^{\frac{1}{2\alpha+2}}x
\right|^{2\alpha+2}\right)
\mathrm{d}x
\\&\qquad+\int_{B_{R}^{c}}
|\nabla u|
\left|b_j^{\frac{1}{2\alpha+2}}
x\right|^{2\alpha+1}
\exp\left(-\frac{1}{2\alpha+2}
\left|b_j^{\frac{1}{2\alpha+2}}x
\right|^{2\alpha+2}\right)
\mathrm{d}x
\\&\quad\leq%2
\epsilon\left[\int_{B_{R}}
\left|b_j^{\frac{1}{2\alpha+2}}
x\right|^{4\alpha+2}
\exp\left(-\frac{1}{\alpha+1}
\left|b_j^{\frac{1}{2\alpha+2}}x
\right|^{2\alpha+2}\right)
\mathrm{d}x\right]^{\frac{1}{2}}
\\&\qquad
+\left(\int_{B_{R}^{c}}
|\nabla u|^{2}\mathrm{d}x\right)  ^{\frac{1}{2}}
\left[\int_{B_{R}^{c}}
\left|b_j^{\frac{1}{2\alpha+2}}
x\right|^{4\alpha+2}
\exp\left(-\frac{1}{\alpha+1}
\left|b_j^{\frac{1}{2\alpha+2}}x
\right|^{2\alpha+2}\right)
\mathrm{d}x\right]^{\frac{1}{2}}
\\&\quad\leq%3
\epsilon\left[\int_{B_{R}}
\left|b_j^{\frac{1}{2\alpha+2}}
x\right|^{4\alpha+2}
\exp\left(-\frac{1}{\alpha+1}
\left|b_j^{\frac{1}{2\alpha+2}}x
\right|^{2\alpha+2}\right)
\mathrm{d}x\right]^{\frac{1}{2}}
\\&\qquad
+\left[\int_{B_{R}^{c}}
\left|b_j^{\frac{1}{2\alpha+2}}
x\right|^{4\alpha+2}
\exp\left(-\frac{1}{\alpha+1}
\left|b_j^{\frac{1}{2\alpha+2}}x
\right|^{2\alpha+2}\right)
\mathrm{d}x\right]^{\frac{1}{2}}
\\&\quad=%4
b_j^{-\frac{N}{4\alpha+4}}
\epsilon\left[\int_{B_{\bar{R}}}
\left|x\right|^{4\alpha+2}
\exp\left(-\frac{1}{\alpha+1}
\left|x\right|^{2\alpha+2}\right)
\mathrm{d}x\right]^{\frac{1}{2}}
\\&\qquad+b_j^{-\frac{N}{4\alpha+4}}
\left[
\int_{B_{\bar{R}}^{c}}
\left|x\right|^{4\alpha+2}
\exp\left(-\frac{1}{\alpha+1}
\left|x\right|^{2\alpha+2}\right)
\mathrm{d}x\right]^{\frac{1}{2}},
\end{align*}
where $\bar{R}
:=Rb_j^{\frac{1}{2\alpha+2}}$, the above inequality together with \eqref{relationab} implies that
\begin{align*}
&\left|a_j
b_j^{\frac{1}{2\alpha+2}}
\int_{\mathbb{R}^{N}}\nabla u\cdot\left(b_j^{\frac{2\alpha+1}{2\alpha+2}}
|x|^{2\alpha }x\right)
\exp\left(-\frac{1}{2\alpha+2}
\left|b_j^{\frac{1}{2\alpha+2}}x
\right|^{2\alpha+2}\right)
\mathrm{d}x\right|
\\&\quad\leq
M\epsilon\left[
\int_{B_{\bar{R}}}
\left|x\right|^{4\alpha+2}
\exp\left(-\frac{1}{\alpha+1}
\left|x\right|^{2\alpha+2}\right)
\mathrm{d}x\right]^{\frac{1}{2}}
\\&\qquad+M\left[
\int_{B_{\bar{R}}^{c}}
\left|x\right|^{4\alpha+2}
\exp\left(-\frac{1}{\alpha+1}
\left|x\right|^{2\alpha+2}\right)
\mathrm{d}x\right]^{\frac{1}{2}}.
\end{align*}
Let $j\rightarrow\infty$, we obtain
\begin{align*}
&\limsup_{j\rightarrow\infty}\left|a_j
b_j^{\frac{1}{2\alpha+2}}
\int_{\mathbb{R}^{N}}\nabla u\cdot\left(b_j^{\frac{2\alpha+1}{2\alpha+2}}
|x|^{2\alpha }x\right)
\exp\left(-\frac{1}{2\alpha+2}
\left|b_j^{\frac{1}{2\alpha+2}}x
\right|^{2\alpha+2}\right)
\mathrm{d}x\right|
\\&\quad\leq M\epsilon\left[ \int_{\mathbb{R}^{N}}
\left|x\right|^{4\alpha+2}
\exp\left(-\frac{1}{\alpha+1}
\left|x\right|^{2\alpha+2}\right)
\mathrm{d}x\right]^{\frac{1}{2}}.
\end{align*}
Due to the arbitrariness of $\epsilon>0$, we have
\[
\limsup_{j\rightarrow\infty}\left|a_j
b_j^{\frac{1}{2\alpha+2}}
\int_{\mathbb{R}^{N}}\nabla u\cdot\left(b_j^{\frac{2\alpha+1}{2\alpha+2}}
|x|^{2\alpha }x\right)
\exp\left(-\frac{1}{2\alpha+2}
\left|b_j^{\frac{1}{2\alpha+2}}x
\right|^{2\alpha+2}\right)
\mathrm{d}x\right|=0,
\]
and then
\begin{equation*}
0=\lim_{j\rightarrow\infty}
\|\nabla(u-v_{j})\|_{2}^2>1,
\end{equation*}
which is a contradiction. Hence, Case (1) can not happen.

If Case (2) holds, that is,  $b_j\rightarrow0$ as $j\rightarrow\infty$. For any
$\epsilon>0$, there exists a constant $R>0$ large enough such that
$\int_{B_{R}^{c}}|\nabla u|^{2}\mathrm{d}x<\epsilon^{2}$. Thus,
\begin{align*}
& \left|\int_{\mathbb{R}^{N}}\nabla u\cdot\left(b_j^{\frac{2\alpha+1}{2\alpha+2}}
|x|^{2\alpha }x\right)
\exp\left(-\frac{1}{2\alpha+2}
\left|b_j^{\frac{1}{2\alpha+2}}x
\right|^{2\alpha+2}\right)
\mathrm{d}x\right|
\\&\quad\leq%1
\int_{B_{R}^{c}}
|\nabla u|
\left|b_j^{\frac{1}{2\alpha+2}}
x\right|^{2\alpha+1}
\exp\left(-\frac{1}{2\alpha+2}
\left|b_j^{\frac{1}{2\alpha+2}}x
\right|^{2\alpha+2}\right)
\mathrm{d}x
\\&\qquad+\int_{B_{R}}
|\nabla u|
\left|b_j^{\frac{1}{2\alpha+2}}
x\right|^{2\alpha+1}
\exp\left(-\frac{1}{2\alpha+2}
\left|b_j^{\frac{1}{2\alpha+2}}x
\right|^{2\alpha+2}\right)
\mathrm{d}x
\\&\quad\leq%2
\epsilon\left[\int_{B_{R}^{c}}
\left|b_j^{\frac{1}{2\alpha+2}}
x\right|^{4\alpha+2}
\exp\left(-\frac{1}{\alpha+1}
\left|b_j^{\frac{1}{2\alpha+2}}x
\right|^{2\alpha+2}\right)\mathrm{d}x\right]  ^{\frac{1}{2}}
\\&\qquad
+\left(\int_{B_{R}}
|\nabla u|^{2}\mathrm{d}x\right)  ^{\frac{1}{2}}
\left[\int_{B_{R}}
\left|b_j^{\frac{1}{2\alpha+2}}
x\right|^{4\alpha+2}
\exp\left(-\frac{1}{\alpha+1}
\left|b_j^{\frac{1}{2\alpha+2}}x
\right|^{2\alpha+2}\right)\mathrm{d}x\right]  ^{\frac{1}{2}}
\\&\quad\leq%3
\epsilon\left[\int_{B_{R}^{c}}
\left|b_j^{\frac{1}{2\alpha+2}}
x\right|^{4\alpha+2}
\exp\left(-\frac{1}{\alpha+1}
\left|b_j^{\frac{1}{2\alpha+2}}x
\right|^{2\alpha+2}\right)\mathrm{d}x\right]  ^{\frac{1}{2}}
\\&\qquad
+\left[\int_{B_{R}}
\left|b_j^{\frac{1}{2\alpha+2}}
x\right|^{4\alpha+2}
\exp\left(-\frac{1}{\alpha+1}
\left|b_j^{\frac{1}{2\alpha+2}}x
\right|^{2\alpha+2}\right)
\mathrm{d}x\right]^{\frac{1}{2}}
\\&\quad=%4
b_j^{-\frac{N}{4\alpha+4}}
\epsilon\left[\int_{B_{\bar{R}}^{c}}
\left|x\right|^{4\alpha+2}
\exp\left(-\frac{1}{\alpha+1}
\left|x\right|^{2\alpha+2}\right)
\mathrm{d}x\right]^{\frac{1}{2}}
\\&\qquad+b_j^{-\frac{N}{4\alpha+4}}
\left[
\int_{B_{\bar{R}}}
\left|x\right|^{4\alpha+2}
\exp\left(-\frac{1}{\alpha+1}
\left|x\right|^{2\alpha+2}\right)
\mathrm{d}x\right]^{\frac{1}{2}},
\end{align*}
this together with \eqref{relationab} yields that
\begin{align*}
&\left|a_j
b_j^{\frac{1}{2\alpha+2}}
\int_{\mathbb{R}^{N}}\nabla u\cdot\left(b_j^{\frac{2\alpha+1}
{2\alpha+2}}
|x|^{2\alpha }x\right)
\exp\left(-\frac{1}{2\alpha+2}
\left|b_j^{\frac{1}{2\alpha+2}}x
\right|^{2\alpha+2}\right)
\mathrm{d}x\right|
\\&\quad\leq
M\epsilon\left[
\int_{B_{\bar{R}}^{c}}
\left|x\right|^{4\alpha+2}
\exp\left(-\frac{1}{\alpha+1}
\left|x\right|^{2\alpha+2}\right)
\mathrm{d}x\right]^{\frac{1}{2}}
\\&\qquad+M\left[
\int_{B_{\bar{R}}}
\left|x\right|^{4\alpha+2}
\exp\left(-\frac{1}{\alpha+1}
\left|x\right|^{2\alpha+2}\right)
\mathrm{d}x\right]^{\frac{1}{2}}.
\end{align*}
Let $j\rightarrow\infty$, we obtain
\begin{align*}
&\limsup_{j\rightarrow\infty}\left|a_j
b_j^{\frac{1}{2\alpha+2}}
\int_{\mathbb{R}^{N}}\nabla u\cdot\left(b_j^{\frac{2\alpha+1}{2\alpha+2}}
|x|^{2\alpha }x\right)
\exp\left(-\frac{1}{2\alpha+2}
\left|b_j^{\frac{1}{2\alpha+2}}x
\right|^{2\alpha+2}\right)
\mathrm{d}x\right|
\\&\quad\leq M\epsilon\left[ \int_{\mathbb{R}^{N}}
\left|x\right|^{4\alpha+2}
\exp\left(-\frac{1}{\alpha+1}
\left|x\right|^{2\alpha+2}\right)
\mathrm{d}x\right]^{\frac{1}{2}}.
\end{align*}
Since $\epsilon>0$ is arbitrary, then
\[
\limsup_{j\rightarrow\infty}\left|a_j
b_j^{\frac{1}{2\alpha+2}}
\int_{\mathbb{R}^{N}}\nabla u\cdot\left(b_j^{\frac{2\alpha+1}{2\alpha+2}}
|x|^{2\alpha }x\right)
\exp\left(-\frac{1}{2\alpha+2}
\left|b_j^{\frac{1}{2\alpha+2}}x
\right|^{2\alpha+2}\right)
\mathrm{d}x\right|=0,
\]
and then
\begin{equation*}
0=\lim_{j\rightarrow\infty}
\|\nabla(u-v_{j})\|_{2}^2>1,
\end{equation*}
which also leads to a contradiction. Hence, Case (2) also can not happen.

In view of the above argument, both Case (1) and Case (2) can not happen. Therefore, Case (3) happens, then  there exists $b\in(0,\infty)$ such that $\lim_{j\rightarrow\infty}b_j=b$. This combining with \eqref{relationab} yields that $\{a_{j}\}_j$ is bounded. Here the proof of the claim is completed.
\end{proof}

\subsection{Proof of main results}

Based on the above preparations, we are in a position to prove our main results.

\begin{proof}
[\rm\textbf{Proof of Theorem \ref{thm-1}}]
For simplicity of notation, let us denote
\[
K(N,\alpha,k)
:=\sqrt{\left(N+2\alpha\right)^{2}
+4(N-2+k)k}-(N+2\alpha).
\]
To verify the main result of Theorem \ref{thm-1}, it reduces to show $C(N,\alpha,k)=K(N,\alpha,k)$.

$\bullet$ \emph{Step 1: For all $k\in\mathbb{N}^+$, we will verify  that}
\begin{equation}\label{9.19-1}
C(N,\alpha,k)
\geq K(N,\alpha,k).
\end{equation}
Indeed, for constant $K>0$ (the range of this parameter will be provided later), we get
\begin{align}\label{9.19-5}
&\int_{\mathbb{R}^{N+2k}}
\frac{|\Delta v_k(|x|)|^{2}}
{|x|^{2\alpha}}
\mathrm{d}x
+\int_{\mathbb{R}^{N+2k}}
\left[|x|^{2\alpha+2}
|\nabla v_k(|x|)|^{2}
-2(\alpha+1)k
|x|^{2\alpha}|v_k(|x|)|^{2}\right]
\mathrm{d}x
\nonumber\\&\quad  -(N+4\alpha+2+K)\int_{\mathbb{R}^{N+2k}}
|\nabla v_k(|x|)|^{2}\mathrm{d}x
\nonumber\\&\quad=%1
\int_{0}^{\infty}\left|rv''_k
+(N+2k-1)v_k'\right|^{2}
r^{N+2k-2\alpha-3}\mathrm{d}r
+\int_{0}^{\infty}
\left|v_k'\right|^{2}
r^{N+2k+2\alpha+1}
\mathrm{d}r
\nonumber\\&\qquad
-(N+4\alpha+2+K)
\int_{0}^{\infty}
\left|v_k'\right|^{2}
r^{N+2k-1}
\mathrm{d}r
-2(\alpha+1)k
\int_{0}^{\infty}
\left|v_k\right|^{2}
r^{N+2k+2\alpha-1}
\mathrm{d}r
\nonumber\\&\quad=%2
\int_{0}^{\infty}\left|rv''_k
+(N+2k-1)v_k'+r^{2\alpha+2}v_k'
+\frac{2N+2k+4\alpha+K}{2}
r^{2\alpha+1}v_k\right|^{2}
r^{N+2k-2\alpha-3}\mathrm{d}r
\nonumber\\&\qquad
-\left[
\frac{(2N+2k+4\alpha+K)(K-2k)}{4}
+2(\alpha+1)k
\right]
\int_{0}^{\infty}
\left|v_k\right|^{2}
r^{N+2k+2\alpha -1}
\mathrm{d}r.
\end{align}
Then, by choosing $K>0$ satisfying
\[
\frac{(2N+2k+4\alpha+K)
(K-2k)}{4}
+2(\alpha+1)k
\leq0,
\]
which is equivalent to
\begin{equation*}
K^{2}+2(N+2\alpha)K
-4\left(N-2+k\right)k
\leq0.
\end{equation*}
That is, $0<K\leq K(N,\alpha,k)$. Thereby, for some constant $0<K\leq K(N,\alpha,k)$, it follows from \eqref{9.19-5} that
\begin{align*}
&\int_{\mathbb{R}^{N+2k}}
\frac{|\Delta v_k(|x|)|^{2}}
{|x|^{2\alpha}}
\mathrm{d}x
+\int_{\mathbb{R}^{N+2k}}
\left[|x|^{2\alpha+2}
|\nabla v_k(|x|)|^{2}
-2(\alpha+1)k
|x|^{2\alpha}|v_k(|x|)|^{2}\right]
\mathrm{d}x
\nonumber\\&\quad
\ge(N+4\alpha+2+K)
\int_{\mathbb{R}^{N+2k}}
|\nabla v_k(|x|)|^{2}\mathrm{d}x.
\end{align*}
In view of this, it yields from the definition of $C(N,\alpha,k)$ (see \eqref{9.18-3}) that \eqref{9.19-1} holds.

$\bullet$ \emph{Step 2: For all $k\in \mathbb{N}^+$, we will verify that}
\begin{equation}\label{9.19-2}
C(N,\alpha,k)\leq K(N,\alpha,k).
\end{equation}
Choosing $K=K(N,\alpha,k)$ in \eqref{9.19-5}, and then
\begin{align*}
&\int_{\mathbb{R}^{N+2k}}
\frac{|\Delta v_k(|x|)|^{2}}
{|x|^{2\alpha}}
\mathrm{d}x
+\int_{\mathbb{R}^{N+2k}}
\left[|x|^{2\alpha+2}
|\nabla v_k(|x|)|^{2}
-2(\alpha+1)k
|x|^{2\alpha}|v_k(|x|)|^{2}\right]
\mathrm{d}x
\nonumber\\&\qquad  -\left[N+4\alpha+2+K(N,\alpha,k)\right]
\int_{\mathbb{R}^{N+2k}}
|\nabla v_k(|x|)|^{2}\mathrm{d}x
\equiv 0
\end{align*}
if and only if $v_k$ is the solution of
\begin{equation}\label{9.21-1}
rv''_k
+(N+2k-1)v_k'
+r^{2\alpha+2}v_k'
+\frac{2N+2k+4\alpha+K(N,\alpha,k)}{2}
r^{2\alpha+1}v_k=0.
\end{equation}
By Lemma \ref{lem-3.1}, the equation \eqref{9.21-1} has the two independent solutions
{\small
\begin{align*}
v_{k,1}(r)&={_1F_1}
\left(\frac{2N+2k+4\alpha+K(N,\alpha,k)}
{4\alpha+4};
\frac{N+2k+2\alpha}{2\alpha+2};
-\frac{r^{2\alpha+2}}{2\alpha+2}\right),
\\
v_{k,2}(r)&=
\left(-\frac{r^{2\alpha+2}}
{2\alpha+2}\right)^{
-\frac{N+2k-2}{2\alpha+2}}
{_1F_1}
\left(\frac{-2k+4\alpha+4+K(N,\alpha,k)}
{4\alpha+4};
1-\frac{N+2k-2}{2\alpha+2};
-\frac{r^{2\alpha+2}}{2\alpha+2}\right).
\end{align*}}\noindent
Moreover, $v_{k,1}(r)\in C^\infty(\mathbb{R}^+)$ and by \eqref{9.19-7}, for all $l\in\mathbb{N}$,
\begin{align*}
v_{k,1}^{(l)}(r)
&=O\left(r^{
-N-k-2\alpha-\frac{K(N,\alpha,k)}{2}-l}
\right),
\ \
\mathrm{as}
\
r\to\infty.
\end{align*}
Meanwhile, since
\[
v''_{k,1}(r)
=cr^{2\alpha}+o(r^{2\alpha}),\quad v''_{k,2}(r)
=cr^{-N-2k}+o(r^{-N-2k}),
\ \
\mathrm{as}
\
r\to 0,
\]
then
\begin{align*}
\int_{0}^{\infty}
|rv''_{k,1}|^{2}
r^{N+2k-2\alpha-3}\mathrm{d}r
&=\int_{0}^{\infty}
|v''_{k,1}|^{2}
r^{N+2k-2\alpha-1}\mathrm{d}r
<\infty,
\\
\int_{0}^{\infty}
|rv''_{k,2}|^{2}
r^{N+2k-2\alpha-3}\mathrm{d}r
&=\int_{0}^{\infty}
|v''_{k,2}|^{2}
r^{N+2k-2\alpha-1}\mathrm{d}r
=\infty.
\end{align*}
Thus, $r^kv_{k,1}(r)\phi_k(\sigma)$ belongs to $X$, while $r^kv_{k,2}(r)\phi_k(\sigma)$ can not belong to $X$. Therefore, the equality in \eqref{9.21-1} holds if and only if
\begin{equation*}
v_{k,1}(r)=
\,{_1F_1}
\left(\frac{2N+2k+4\alpha+K(N,\alpha,k)}
{4\alpha+4};
\frac{N+2k+2\alpha}{2\alpha+2};
-\frac{r^{2\alpha+2}}{2\alpha+2}\right).
\end{equation*}
In view of this, for $K>K(N,\alpha,k)$ and $k\ge1$,
\begin{align*}
&\int_{\mathbb{R}^{N+2k}}
\frac{|\Delta v_{k,1}|^{2}}
{|x|^{2\alpha}}
\mathrm{d}x
+\int_{\mathbb{R}^{N+2k}}
\left[|x|^{2\alpha+2}
|\nabla v_{k,1}|^{2}
-2(\alpha+1)k
|x|^{2\alpha}|v_{k,1}|^{2}\right]
\mathrm{d}x
\nonumber\\&\qquad  -(N+4\alpha+2+K)\int_{\mathbb{R}^{N+2k}}
|\nabla v_{k,1}|^{2}\mathrm{d}x
\\&\quad=
-\left[
\frac{(2N+2k+4\alpha+K)(K-2k)}{4}
+2(\alpha+1)k\right]
\int_{0}^{\infty}
\left|v_{k,1}\right|^{2}
r^{N+2k+2\alpha-1}
\mathrm{d}r
\\&\quad<0.
\end{align*}
Therefore, from the definition of $C(N,\alpha,k)$ (see \eqref{9.18-3}), the desired estimate \eqref{9.19-2} holds.

$\bullet$ \emph{Step 3: Conclusion.}

Combining \eqref{9.19-1} with  \eqref{9.19-2}, we see that
\begin{equation*}
C(N,\alpha,k)
=K(N,\alpha,k)
=\sqrt{\left(N+2\alpha\right)^{2}
+4(N-2+k)k}-(N+2\alpha).
\end{equation*}
Based on the assumptions $N\ge2$ and  $\alpha>-1$, it is easy to check that $C(N,\alpha,k)$ is a increasing function with respect to $k\ge1$. Thus, for all $k\ge1$,
\begin{align*}
C(N,\alpha,k)\ge C(N,\alpha,1),
\end{align*}
and then
\begin{equation}\label{9.19-10}
\inf_{k\in\mathbb{N}^+}C(N,\alpha,k)
=C(N,\alpha,1)
=\sqrt{\left(N+2\alpha\right)^{2}
+4(N-1)}-(N+2\alpha).
\end{equation}
Hence, the proof is completed.
\end{proof}

\begin{proof}[\rm\textbf{Proof of Theorem \ref{thm-2}}]
Here we aim to determine the constant $K>0$ such that
\begin{align}\label{2ndOrdStab}
&\int_{\mathbb{R}^{N}}
\frac{|\Delta u|^{2}}{|x|^{2\alpha}}
\mathrm{d}x
+\int_{\mathbb{R}^{N}}
\left|x\right|^{2\alpha+2}
|\nabla u|^{2}\mathrm{d}x
-\left(N+4\alpha+2\right)
\int_{\mathbb{R}^{N}}
|\nabla u|^{2}\mathrm{d}x
\nonumber\\&\quad\geq K\inf_{c}\int_{\mathbb{R}^{N}}
\left|\nabla\left[
u-c\exp\left(-\frac{|x|^{2\alpha+2}}
{2\alpha+2}\right)\right] \right|^{2}\mathrm{d}x.
\end{align}
From \eqref{9.12-2}--\eqref{9.12-4}, we have
\begin{align}\label{9.21-3}
&\int_{\mathbb{R}^{N}}
\frac{|\Delta u|^{2}}
{|x|^{2\alpha}}
\mathrm{d}x
+\int_{\mathbb{R}^{N}}
|x|^{2\alpha+2}
|\nabla u|^{2}\mathrm{d}x
-\left(N+4\alpha+2\right)
\int_{\mathbb{R}^{N}}
|\nabla u|^{2}\mathrm{d}x
\nonumber\\&\quad%1
=\int_0^\infty
|w'_0|^{2}
r^{N+2\alpha+1}\mathrm{d}r
+\int_0^\infty
|w_0|^{2}
r^{N+6\alpha+3}
\mathrm{d}r
-(N+4\alpha+2)
\int_0^\infty
|w_0|^{2}r^{N+4\alpha+1}
\mathrm{d}r
\nonumber\\&\qquad
+\sum_{k=1}^{\infty}
\int_{\mathbb{R}^{N+2k}}
\frac{|\Delta v_k|^{2}}
{|x|^{2\alpha}}
\mathrm{d}x
+\sum_{k=1}^{\infty}
\int_{\mathbb{R}^{N+2k}}
\left[|x|^{2\alpha+2}
|\nabla v_k|^{2}
-2(\alpha+1)k
|x|^{2\alpha}|v_k|^{2}\right]
\mathrm{d}x
\nonumber\\&\qquad -(N+4\alpha+2)
\sum_{k=1}^{\infty}
\int_{\mathbb{R}^{N+2k}}
|\nabla v_k|^{2}\mathrm{d}x,
\end{align}
where $u(x)=u(r\sigma)
=\sum_{k=0}^{\infty}
r^{k}v_{k}(r)
\phi_{k}(\sigma)$ and $w_0(r)
=\dfrac{v_0^{\prime}(r)}
{r^{2\alpha+1}}$.

Meanwhile,
\begin{align}\label{9.14-1}
&\inf_{c}\int_{\mathbb{R}^{N}}
\left|\nabla\left[
u-c\exp\left(-\frac{|x|^{2\alpha+2}}
{2\alpha+2}\right)\right]  \right|^{2}\mathrm{d}x
\nonumber\\&\quad=
\int_{\mathbb{R}^{N}}
\left|\nabla u- \dfrac{\int_{\mathbb{R}^{N}}
\nabla u\cdot\nabla\left[
\exp\left(-\frac{|x|^{2\alpha+2}}
{2\alpha+2}\right)\right]
\mathrm{d}x}
{\int_{\mathbb{R}^{N}}
\left|\nabla
\left[\exp\left(-\frac{|x|^{2\alpha+2}}
{2\alpha+2}\right)\right]  \right|^{2}\mathrm{d}x} \nabla\left[\exp
\left(-\frac{|x|^{2\alpha+2}}
{2\alpha+2}\right)\right]  \right|^{2}\mathrm{d}x
\nonumber\\&\quad=%2
\int_{\mathbb{R}^{N}}
\left|\nabla u\right|^2\mathrm{d}x
-\dfrac{\left(\int_{\mathbb{R}^{N}}
\nabla u\cdot\nabla\left[
\exp\left(-\frac{|x|^{2\alpha+2}}
{2\alpha+2}\right)\right]
\mathrm{d}x\right)^2}
{\int_{\mathbb{R}^{N}}
\left|\nabla
\left[\exp\left(-\frac{|x|^{2\alpha+2}}
{2\alpha+2}\right)\right]  \right|^{2}\mathrm{d}x}
\nonumber\\&\quad=%3
\int_0^\infty
|w_0|^{2}r^{N+4\alpha+1}
\mathrm{d}r
+\sum_{k=1}^{\infty}
\int_{\mathbb{R}^{N+2k}}
|\nabla v_k|^{2}
\mathrm{d}x
\nonumber\\&\qquad
-\dfrac{\left(\int_{\mathbb{R}^{N}}
\nabla u\cdot\nabla\left[
\exp\left(-\frac{|x|^{2\alpha+2}}
{2\alpha+2}\right)\right]
\mathrm{d}x\right)^2}
{\int_{\mathbb{R}^{N}}
\left|\nabla
\left[\exp\left(-\frac{|x|^{2\alpha+2}}
{2\alpha+2}\right)\right]  \right|^{2}\mathrm{d}x}.
\end{align}
Note that
\begin{align}\label{9.14-2}
&\int_{\mathbb{R}^{N}}
\nabla u\cdot\nabla\left[
\exp\left(-\frac{|x|^{2\alpha+2}}
{2\alpha+2}\right)\right]
\mathrm{d}x
\nonumber\\&\quad=%1
-\int_{\mathbb{R}^{N}}u
\Delta\left[
\exp\left(-\frac{|x|^{2\alpha+2}}
{2\alpha+2}\right)\right]
\mathrm{d}x
\nonumber\\&\quad=%2
-\int_{\mathbb{R}^{N}}
\left[\sum_{k=0}^{\infty}
v_{k}(|x|)|x|^{k}
\phi_{k}(\sigma)\right]
\Delta\left[
\exp\left(-\frac{|x|^{2\alpha+2}}
{2\alpha+2}\right)\right]
\mathrm{d}x
\nonumber\\&\quad=%3
-\sum_{k=0}^{\infty}
\int_{\mathbb{R}^{N}}
v_{k}(|x|)|x|^{k}
\phi_{k}(\sigma)
\Delta\left[
\exp\left(-\frac{|x|^{2\alpha+2}}
{2\alpha+2}\right)\right]
\mathrm{d}x
\nonumber\\&\quad=%4
-\int_{\mathbb{R}^{N}}
v_{0}(|x|)
\phi_{0}(\sigma)
\Delta\left[
\exp\left(-\frac{|x|^{2\alpha+2}}
{2\alpha+2}\right)\right]
\mathrm{d}x
\nonumber\\&\quad=%5
\int_{\mathbb{R}^{N}}
\nabla v_{0}(|x|)
\cdot\nabla\left[
\exp\left(-\frac{|x|^{2\alpha+2}}
{2\alpha+2}\right)\right]
\mathrm{d}x
\nonumber\\&\quad=%6
-\int_0^\infty
v'_0(r)\exp\left(-\frac{r^{2\alpha+2}}
{2\alpha+2}\right)
r^{N+2\alpha}\mathrm{d}r
\nonumber\\&\quad=%7
-\int_0^\infty
w_0(r)\exp\left(-\frac{r^{2\alpha+2}}
{2\alpha+2}\right)
r^{N+4\alpha+1}\mathrm{d}r,
\end{align}
with the aid of  $\int_{\mathbb{S}^{N-1}}
\phi_{k}(\sigma)
\mathrm{d}\sigma=0$ for all
$k\geq1$, and $\phi_{0}(\sigma)\equiv1$. Then, substituting \eqref{9.14-2} into \eqref{9.14-1}, there holds
\begin{align}\label{9.21-2}
&\inf_{c}\int_{\mathbb{R}^{N}}
\left|\nabla\left[
u-c\exp\left(-\frac{|x|^{2\alpha+2}}
{2\alpha+2}\right)\right]  \right|^{2}\mathrm{d}x
\nonumber\\&\quad=%1
\int_0^\infty
|w_0|^{2}r^{N+4\alpha+1}\mathrm{d}r
+\sum_{k=1}^{\infty}
\int_{\mathbb{R}^{N+2k}}
|\nabla v_k|^{2}\mathrm{d}x
-\frac{\left[
\int_0^\infty
w_{0}
\exp\left(-\frac{r^{2\alpha+2}}
{2\alpha+2}\right)r^{N+4\alpha+1}
\mathrm{d}r\right]^2}
{\int_0^\infty
\exp\left(-\frac{r^{2\alpha+2}}
{\alpha+1}\right)
r^{N+4\alpha+1}\mathrm{d}r}
\nonumber\\&\quad%2
=\inf_{c}\int_0^\infty
\left|w_0-c
\exp\left(-\frac{r^{2\alpha+2}}
{2\alpha+2}\right)\right|^2
r^{N+4\alpha+1}\mathrm{d}r
+\sum_{k=1}^{\infty}
\int_{\mathbb{R}^{N+2k}}
|\nabla v_k|^{2}\mathrm{d}x.
\end{align}
Hence, combining \eqref{9.21-3} with \eqref{9.21-2}, \eqref{2ndOrdStab} is equivalent to
\begin{align}\label{9.19-11}
&\int_0^\infty
|w'_0|^{2}
r^{N+2\alpha+1}\mathrm{d}r
+\int_0^\infty
|w_0|^{2}
r^{N+6\alpha+3}
\mathrm{d}r
-(N+4\alpha+2)
\int_0^\infty
|w_0|^{2}r^{N+4\alpha+1}
\mathrm{d}r
\nonumber\\&\qquad
+\sum_{k=1}^{\infty}
\int_{\mathbb{R}^{N+2k}}
\frac{|\Delta v_k|^{2}}
{|x|^{2\alpha}}\mathrm{d}x
+\sum_{k=1}^{\infty}
\int_{\mathbb{R}^{N+2k}}
\left[|x|^{2\alpha+2}
|\nabla v_k|^{2}
-2(\alpha+1)k
|x|^{2\alpha}|v_k|^{2}\right]
\mathrm{d}x
\nonumber\\&\qquad -(N+4\alpha+2+K)
\sum_{k=1}^{\infty}
\int_{\mathbb{R}^{N+2k}}
|\nabla v_k|^{2}\mathrm{d}x
\nonumber\\&\quad\ge
K\inf_{c}\int_0^\infty
\left|w_0-c
\exp\left(-\frac{r^{2\alpha+2}}
{2\alpha+2}\right)\right|^2
r^{N+4\alpha+1}\mathrm{d}r.
\end{align}
Thus, in order to obtain \eqref{9.19-11}, it suffices to verify that
\begin{align}\label{N1}
&\int_0^\infty
|w'_0|^{2}
r^{N+2\alpha+1}\mathrm{d}r
+\int_0^\infty
|w_0|^{2}
r^{N+6\alpha+3}
\mathrm{d}r
-(N+4\alpha+2)
\int_0^\infty
|w_0|^{2}r^{N+4\alpha+1}
\mathrm{d}r
\nonumber\\&\quad\ge%2
K\inf_{c}\int_0^\infty
\left|w_0-c
\exp\left(-\frac{r^{2\alpha+2}}
{2\alpha+2}\right)\right|^2
r^{N+4\alpha+1}\mathrm{d}r,
\end{align}
and
\begin{align}\label{N2}
&\int_{\mathbb{R}^{N+2k}}
\frac{|\Delta v_k|^{2}}
{|x|^{2\alpha}}
\mathrm{d}x
+\int_{\mathbb{R}^{N+2k}}
\left[|x|^{2\alpha+2}
|\nabla v_k|^{2}
-2(\alpha+1)k
|x|^{2\alpha}|v_k|^{2}\right]
\mathrm{d}x
\nonumber\\&\quad
\ge(N+4\alpha+2+K)
\int_{\mathbb{R}^{N+2k}}
|\nabla v_k|^{2}\mathrm{d}x,
\end{align}
for all $k\ge1$. Hence, it yields from Corollary \ref{coro-5.2} (by replacing $m$ into $\frac{N+4\alpha+2}{\alpha+1}$) that
\begin{align}\label{9.22-2}
&\int_0^\infty
|w'_0|^{2}
r^{N+2\alpha+1}\mathrm{d}r
+\int_0^\infty
|w_0|^{2}
r^{N+6\alpha+3}
\mathrm{d}r
-(N+4\alpha+2)
\int_0^\infty
|w_0|^{2}r^{N+4\alpha+1}
\mathrm{d}r
\nonumber\\&\quad=%1
\int_0^\infty
\left|w'_0
+w_0r^{2\alpha+1}
\right|^2r^{N+2\alpha+1}\mathrm{d}r
\nonumber\\&\quad=%2
\int_0^\infty
\left|
\left[w_0\exp\left(\frac{r^{2\alpha+2}}
{2\alpha+2}\right)
\right]'\right|^2
\exp\left(-\frac{r^{2\alpha+2}}
{\alpha+1}\right)
r^{N+2\alpha+1}\mathrm{d}r
\nonumber\\&\quad\ge%3
4(\alpha+1)
\inf_{c}
\int_0^\infty
\left|w_0\exp\left(\frac{r^{2\alpha+2}}
{2\alpha+2}\right)-c\right|^2
\exp\left(-\frac{r^{2\alpha+2}}
{\alpha+1}\right)
r^{N+4\alpha+1}\mathrm{d}r
\nonumber\\&\quad=%4
4(\alpha+1)
\inf_{c}
\int_0^\infty
\left|w_0-c\exp
\left(-\frac{r^{2\alpha+2}}
{2\alpha+2}\right)\right|^2
r^{N+4\alpha+1}\mathrm{d}r,
\end{align}
Therefore, this together with \eqref{N1} implies that
\[
0<K\le 4(\alpha+1).
\]
Also, \eqref{N2} is equivalent to
\[
K\leq C(N,\alpha,k).
\]
Therefore,
\[
\delta_{1}(u)
\geq K\inf_{c}\int_{\mathbb{R}^{N}}
\left|\nabla\left[
u-c\exp\left(-\frac{|x|^{2\alpha+2}}
{2\alpha+2}\right)\right] \right|^{2}\mathrm{d}x,
\]
where
\begin{equation*}
K=\min
\left\{\inf_{k\in\mathbb{N}^+}
C(N,\alpha,k),
4(\alpha+1)\right\}
=\min\left\{C(N,\alpha,1),
4(\alpha+1)\right\}
=C(N,\alpha),
\end{equation*}
with the aid of \eqref{9.19-10} and \eqref{defcna}.

Now, we will show that $K=C(N,\alpha)$ is sharp in \eqref{2ndOrdStab}. Indeed, assume by contradiction that there exists $K>C(N,\alpha)$ such that
\begin{align*}
&\int_{\mathbb{R}^{N}}
\frac{|\Delta u|^{2}}{|x|^{2\alpha}}
\mathrm{d}x
+\int_{\mathbb{R}^{N}}
|x|^{2\alpha+2}
|\nabla u|^{2}\mathrm{d}x
-(N+4\alpha+2)\int_{\mathbb{R}^{N}}
|\nabla u|^{2}\mathrm{d}x
\nonumber\\&\quad\geq K\inf_{c}\int_{\mathbb{R}^{N}}
\left|\nabla\left[
u-c\exp\left(-\frac{|x|^{2\alpha+2}}
{2\alpha+2}\right)\right]\right| ^{2}\mathrm{d}x.
\end{align*}
Let $\varepsilon
=\frac{K-C(N,\alpha)}{2}>0$, then we can find $m\in\mathbb{N}^+$ such that
\begin{align*}
C(N,\alpha,m)
&=\inf_{u \text{ is radial}}
\frac{\int_{\mathbb{R}^{N+2m}}
\frac{|\Delta u|^{2}}{|x|^{2\alpha}}
\mathrm{d}x
+\int_{\mathbb{R}^{N+2m}}
\left[|x|^{2\alpha+2}
\left|\nabla u\right|^{2}
\mathrm{d}x
-2(\alpha+1)m
|x|^{2\alpha}|u|^{2}
\right]\mathrm{d}x}
{\int_{\mathbb{R}^{N+2m}}
|\nabla u|^{2}\mathrm{d}x}
\\&\quad-\left(N+4\alpha+2\right)
\\&<C(N,\alpha)+\varepsilon.
\end{align*}
Moreover, there exists a sequence $v_{j}(r)$ such that $\int_{\mathbb{R}
^{N+2m}}|\nabla v_{j}|^{2}\mathrm{d}x=1$ and
\begin{align*}
&\int_{\mathbb{R}^{N+2m}}
\frac{|\Delta v_j|^{2}}{|x|^{2\alpha}}
\mathrm{d}x
+\int_{\mathbb{R}^{N+2m}}
\left[|x|^{2\alpha+2}
\left|\nabla v_j\right|^{2}
\mathrm{d}x
-2(\alpha+1)m
|x|^{2\alpha}|v_j|^{2}
\right]\mathrm{d}x
-\left(N+4\alpha+2\right)
\\&\quad\rightarrow C\left(N,\alpha,m\right),
\ \
j\to\infty.
\end{align*}
Now, choosing $u_{j}=r^{m}v_{j}(r)\phi_{m}(\sigma)$, then by the above calculations, we have
\begin{align*}
&\int_{\mathbb{R}^{N}}
\frac{|\Delta u_j|^{2}}{|x|^{2\alpha}}
\mathrm{d}x
+\int_{\mathbb{R}^{N}}
|x|^{2\alpha+2}
|\nabla u_j|^{2}\mathrm{d}x
-(N+4\alpha+2)\int_{\mathbb{R}^{N}}
|\nabla u_j|^{2}\mathrm{d}x
\nonumber\\&\qquad -K\inf_{c}\int_{\mathbb{R}^{N}}
\left|\nabla\left[
u_j-c\exp\left(-\frac{|x|^{2\alpha+2}}
{2\alpha+2}\right)\right]\right| ^{2}\mathrm{d}x
\nonumber\\&\quad
=\int_{\mathbb{R}^{N+2m}}
\frac{|\Delta v_j|^{2}}{|x|^{2\alpha}}
\mathrm{d}x
+\int_{\mathbb{R}^{N+2m}}
\left[|x|^{2\alpha+2}
\left|\nabla v_j\right|^{2}
\mathrm{d}x
-2(\alpha+1)m
|x|^{2\alpha}|v_j|^{2}
\right]\mathrm{d}x
\\&\qquad
-\left(N+4\alpha+2+K\right)
\nonumber\\&\quad
\to C(N,\alpha,m)-K
\\&\quad
<C(N,\alpha)+\varepsilon-K<0,
\end{align*}
as $j\to\infty$, which is a contradiction. Hence, $C(N,\alpha)$ is the sharp constant of  \eqref{9.22-1}.

Lastly, we will verify the attainability for the sharp constant
$C(N,\alpha)$ of \eqref{9.22-1}. Indeed, if $C(N,\alpha)=C(N,\alpha,1)$,  from the proof of Theorem \ref{thm-1} (with  $k=1$), to reach the sharp equality
\begin{equation}\label{9.22-3}
\delta_{1}(u)
=C(N,\alpha)
\inf_{c}\int_{\mathbb{R}^{N}}
\left|\nabla\left[
u-c\exp\left(-\frac{|x|^{2\alpha+2}}
{2\alpha+2}\right)\right]\right| ^{2}\mathrm{d}x
\end{equation}
if and only if the optimizer $u$ satisfies the form $u(x)=C|x|v_1(|x|)\phi_1(\sigma)$. where
\begin{equation*}
v_1(|x|)
={_1F_1}
\left(\frac{2N+4\alpha+2+C(N,\alpha)}
{4\alpha+4};
\frac{N+2\alpha+2}{2\alpha+2};
-\frac{|x|^{2\alpha+2}}{2\alpha+2}\right).
\end{equation*}

Meanwhile, if $C(N,\alpha)=4(\alpha+1)$, to reach the sharp equality \eqref{9.22-3} if and only if the equality in \eqref{9.22-2} holds, that is,
\[
w_0(r)
=\left(a_0
+\frac{2a_1}{\alpha+1}
r^{2\alpha+2}\right)
\exp\left(-\frac{r^{2\alpha+2}}
{2\alpha+2}\right),
\]
then
\begin{align*}
u(x)
=v_0(|x|)
&=\int_0^{|x|}
r^{2\alpha+1}w_0(r)\mathrm{d}r
\\&=%1
\int_0^{|x|}
r^{2\alpha+1}
\left(a_0+\frac{2a_1}{\alpha+1}
r^{2\alpha+2}\right)
\exp\left(-\frac{r^{2\alpha+2}}
{2\alpha+2}\right)\mathrm{d}r
\\&=%2
a_0\int_0^{|x|}
r^{2\alpha+1}
\exp\left(-\frac{r^{2\alpha+2}}
{2\alpha+2}\right)\mathrm{d}r
+\frac{2a_1}{\alpha+1}\int_0^{|x|}
r^{4\alpha+3}
\exp
\left(-\frac{r^{2\alpha+2}}
{2\alpha+2}\right)\mathrm{d}r
\\&=%3
a_0\int_0^{\frac{|x|^{2\alpha+2}}
{2\alpha+2}}
e^{-r}\mathrm{d}r
+4a_1\int_0^{
\frac{|x|^{2\alpha+2}}{2\alpha+2}}
re^{-r}\mathrm{d}r
\\&=%4
-\frac{2a_1}{\alpha+1}
|x|^{2\alpha+2}
\exp\left(-\frac{|x|^{2\alpha+2}}
{2\alpha+2}\right)
-4a_1
\exp\left(-\frac{|x|^{2\alpha+2}}
{2\alpha+2}\right)
+4a_1
\\&\quad
-a_0\exp\left(-\frac{|x|^{2\alpha+2}}
{2\alpha+2}\right)
+a_0
\\&=%5
-\left(\frac{2a_1}{\alpha+1}
|x|^{2\alpha+2}
+4a_1+a_0\right)
\exp\left(-\frac{|x|^{2\alpha+2}}
{2\alpha+2}\right)
+4a_1+a_0.
\end{align*}
The proof is completed here.
\end{proof}

Based on Theorem \ref{thm-2}, by the scaling argument, we turn to prove Theorem \ref{thm-3}.

\begin{proof}
[\rm\textbf{Proof of Theorem \ref{thm-3}}]
For $\lambda>0$, replacing $u$ by  $u_{\lambda}(x)
=\lambda^{\frac{N-2}{2}}u(\lambda x)$
in \eqref{2ndOrdStab}, we get
\begin{align*}
&\int_{\mathbb{R}^{N}}
\frac{|\Delta u_{\lambda}|^{2}}
{|x|^{2\alpha}}\mathrm{d}x
+\int_{\mathbb{R}^{N}}
\left|x\right|^{2\alpha+2}
|\nabla u_{\lambda}|^{2}\mathrm{d}x
-(N+4\alpha+2)\int
_{\mathbb{R}^{N}}
|\nabla u_{\lambda}|^{2}
\mathrm{d}x
\nonumber\\&\quad\geq
C(N,\alpha)
\inf_{c}\int_{\mathbb{R}^{N}}
\left|\nabla\left[u_{\lambda}
-c\exp\left(
-\frac{\left|x\right|^{2\alpha+2}}
{2\alpha+2}\right)\right]
\right|^{2}\mathrm{d}x,
\end{align*}
which is equivalent to
\begin{align*}
&\lambda^{2\alpha+2}
\int_{\mathbb{R}^{N}}
\frac{|\Delta u|^{2}}{|x|^{2\alpha}}
\mathrm{d}x
+\lambda^{-2\alpha-2}
\int_{\mathbb{R}^{N}}
\left|x\right|^{2\alpha+2}
|\nabla u|^{2}\mathrm{d}x
-(N+4\alpha+2)\int_{\mathbb{R}^{N}}
|\nabla u|^{2}\mathrm{d}x
\nonumber\\&\quad
\geq C(N,\alpha)  \inf_{c}\int_{\mathbb{R}^{N}}
\left|\nabla\left[u
-c\exp\left(
-\frac{\left|x\right|^{2\alpha+2}}
{(2\alpha+2)\lambda^{2\alpha+2}}
\right)\right]
\right|^{2}\mathrm{d}x.
\end{align*}
By choosing
\[
\lambda^{2\alpha+2}=\left(  \dfrac{\int_{\mathbb{R}^{N}}
\left|x\right|^{2\alpha+2}
|\nabla u|^{2}\mathrm{d}x}
{\int_{\mathbb{R}^{N}}
\frac{|\Delta u|^{2}}{|x|^{2\alpha}}
\mathrm{d}x}\right)^{\frac{1}{2}},
\]
we get
\begin{align*}
&2\left(\int_{\mathbb{R}^{N}}
\frac{|\Delta u|^{2}}{|x|^{2\alpha}}
\mathrm{d}x\right)^{\frac{1}{2}}
\left(\int_{\mathbb{R}^{N}}
\left|x\right|^{2\alpha+2}
|\nabla u|^{2}\mathrm{d}x\right)
^{\frac{1}{2}}  -(N+4\alpha+2)\int_{\mathbb{R}^{N}}
|\nabla u|^{2}\mathrm{d}x
\\&\quad\geq C(N,\alpha)  \inf_{c}\int_{\mathbb{R}^{N}}
\left|\nabla\left[u
-c\exp\left(
-\frac{\left|x\right|^{2\alpha+2}}
{(2\alpha+2)\lambda^{2\alpha+2}}
\right)\right]
\right|^{2}\mathrm{d}x.
\end{align*}
Therefore,
\begin{align*}
&\left(\int_{\mathbb{R}^{N}}
\frac{|\Delta u|^{2}}{|x|^{2\alpha}}
\mathrm{d}x\right)  ^{\frac{1}{2}}
\left(\int_{\mathbb{R}^{N}}
\left|x\right|^{2\alpha+2}
|\nabla u|^{2}
\mathrm{d}x\right)^{\frac{1}{2}}
-\frac{N+4\alpha+2}{2}
\int_{\mathbb{R}^{N}}
|\nabla u|^{2}\mathrm{d}x
\\&\quad\geq \dfrac{C(N,\alpha)}{2} \inf_{c}\int_{\mathbb{R}^{N}}
\left|\nabla\left[u
-c\exp\left(
-\frac{\left|x\right|^{2\alpha+2}}
{(2\alpha+2)\lambda^{2\alpha+2}}
\right)\right]
\right|^{2}\mathrm{d}x
\\&\quad
\geq\dfrac{C(N,\alpha)}{2}
\inf_{u^{\ast}\in E_{SHUP}}
\left\|\nabla(u
-u^{\ast})\right\|_{2}^{2}.
\end{align*}
Since $C(N,\alpha)$ is sharp in Theorem \ref{thm-2}, it is easy to see
that $\dfrac{C(N,\alpha)}{2}$ is sharp.
\end{proof}

\begin{proof}
[\rm\textbf{Proof of Theorem \ref{thm-4}}]
Without loss of generality, assume that
$\left\|\nabla u\right\|_{2}^{2}=1$ (otherwise, if
$\left\|\nabla u\right\|_{2}^{2}=0$, the inequality \eqref{cdtStab} holds evidently). We will split the proof into two cases.

\textbf{Case 1:} $\delta_{2}(u)
<\frac{C(N,\alpha)}{2}$. From Lemma \ref{PropOfE}, there exists $v\in E_{SHUP}$ such that
\[
\inf_{u^{\ast}\in E_{SHUP}}
\|\nabla(u-u^{\ast})\|_{2}^{2}
=\|\nabla(u-v)\|_{2}^{2}.
\]
This together with Theorem \ref{thm-3} implies that
\begin{equation}\label{9.18-5}
\|\nabla(u-v)\|^{2}_2
=\inf_{u^{\ast}\in E_{SHUP}}
\|\nabla(u-u^{\ast})\|_{2}^{2}
\leq\dfrac{2}{C(N,\alpha)}
\delta_{2}(u)<1.
\end{equation}
Since $\|\nabla u\|_{2}^{2}=1$, then
$\|\nabla v\|_{2}^{2}\neq0$ (if not, $v$ would be a constant, which is a contradiction). Let $\lambda
=\|\nabla v\|_2^{-1}>0$, and a function $w=\lambda v\in E_{SHUP}$ satisfying $\|\nabla
w\|_{2}^{2}=\lambda^{2}\|\nabla v\|_{2}^{2}=1.$ Hence,
\begin{align*}
&\dfrac{C(N,\alpha)}{2}
\inf_{u^{\ast}\in E_{SHUP}}
\left\{\|\nabla(u
-u^{\ast})\|_{2}^{2}:
\|\nabla u\|_{2}^{2}
=\|\nabla u^{\ast}\|_{2}^{2}\right\}
\\&\quad%1
\leq\dfrac{C(N,\alpha)}{2}
\|\nabla(u-w)\|^{2}_2
\\&\quad%2
=\dfrac{C(N,\alpha)}{2}
\left(2-2\lambda\int_{\mathbb{R}^{N}}
\nabla u\cdot\nabla v\mathrm{d}x\right).
\end{align*}
Notice that
\begin{equation*}
\|\nabla(u-v)\|_2^{2}
=1-2\int_{\mathbb{R}^{N}}
\nabla u\cdot\nabla v\mathrm{d}x
+\dfrac{1}{\lambda^{2}},
\end{equation*}
which together with \eqref{9.18-5} implies that
\[
1-2\int_{\mathbb{R}^{N}}\nabla u\cdot\nabla v\mathrm{d}x
+\dfrac{1}{\lambda^{2}}
\leq\dfrac{2}{C(N,\alpha)}
\delta_{2}(u)<1,
\]
and
\begin{equation}\label{9.18-6}
0<\dfrac{1}{2\lambda^{2}}
\leq\int_{\mathbb{R}^{N}}\nabla u\cdot\nabla v\mathrm{d}x.
\end{equation}
Then, to prove \eqref{cdtStab}, it suffices to prove that
\begin{align*}
2\delta_{2}(u)
&\geq C(N,\alpha)
\left(1-2\int_{\mathbb{R}^{N}}
\nabla u\cdot\nabla v\mathrm{d}x+\dfrac{1}{\lambda^{2}}\right) \\&\geq\dfrac{C(N,\alpha)}{2}
\left(2-2\lambda
\int_{\mathbb{R}^{N}}\nabla u\cdot\nabla v\mathrm{d}x\right),
\end{align*}
which is equivalent to
\begin{equation}\label{9.19-12}
(2-\lambda)\int_{\mathbb{R}^{N}}\nabla u\cdot\nabla v\mathrm{d}x\leq\dfrac{1}
{\lambda^{2}}.
\end{equation}
In fact, by using H\"{o}lder's inequality, we get
\[
0<
\lambda\int_{\mathbb{R}^{N}}
\nabla u\cdot\nabla v\mathrm{d}x
=\int_{\mathbb{R}^{N}}
\nabla u\cdot\nabla w\mathrm{d}x
\le\|\nabla u\|_2\|\nabla w\|_2
=1,
\]
which together with \eqref{9.18-6} yields that
\[
0<\dfrac{1}{2\lambda^{2}}
\leq\int_{\mathbb{R}^{N}}
\nabla u\cdot\nabla v\mathrm{d}x
\leq\dfrac{1}{\lambda}.
\]
Therefore, if $0<\lambda\leq2$,
\[
(2-\lambda)\int_{\mathbb{R}^{N}}
\nabla u\cdot\nabla v\mathrm{d}x
\leq(2-\lambda)
\frac{1}{\lambda}
\leq\dfrac{1}{\lambda^{2}},
\]
if $\lambda>2$,
\begin{equation*}
(2-\lambda)\int_{\mathbb{R}^{N}}
\nabla u\cdot\nabla v\mathrm{d}x
<0<\dfrac{1}{\lambda^{2}},
\end{equation*}
as our desired estimate \eqref{9.19-12}.

\textbf{Case 2:} $\delta_{2}(u)
\geq\frac{C(N,\alpha)}{2}$. In this case, we have
\begin{align*}
&\inf_{u^{\ast}\in E_{SHUP}}
\left\{
\|\nabla(u-u^{\ast})\|_{2}^{2}:
\|\nabla u^{\ast}\|_{2}^{2}
=\|\nabla u\|_{2}^{2}=1\right\}
\\&\quad%1
\leq\dfrac{1}{2}
\inf_{u^{\ast}\in E_{SHUP}}
\left\{\|\nabla(u+u^{\ast})\|_{2}^{2}
+\|\nabla(u-u^{\ast})\|_{2}^{2}:
\|\nabla u^{\ast}\|_{2}^{2}
=\|\nabla u\|_{2}^{2}=1\right\}
\\&\quad%2
=\dfrac{1}{2}
\inf_{u^{\ast}\in E_{SHUP}}
\left\{2\left(\|\nabla u\|_{2}^{2}
+\|\nabla u^{\ast}\|_{2}^{2}\right):
\|\nabla u^{\ast}\|_{2}^{2}
=\|\nabla u\|_{2}^{2}=1\right\}
\\&\quad%3
=2\leq\dfrac{4}{C(N,\alpha)}
\delta_{2}(u),
\end{align*}
which implies that \eqref{cdtStab} holds. This completes the proof of Theorem \ref{thm-4}.
\end{proof}

\section*{Data Availability Statement}
\noindent No data was used for the research described in the article.

\end{document}